\documentclass[12pt,a4paper]{article}

\usepackage[font=small,margin=2cm]{caption}

\usepackage{authblk}
\usepackage[margin=2.3cm]{geometry}
\usepackage{t1enc}
\usepackage[utf8]{inputenc}
\usepackage{amsmath,amsthm,amssymb}
\usepackage{graphicx}
\usepackage{enumerate}
\usepackage{hyperref}
\usepackage{bm}
\usepackage{comment}
\usepackage{amsfonts}
\usepackage{graphicx,caption}
\usepackage{bm}
\usepackage{amsmath, amsthm, amssymb}
\usepackage{graphicx}
\usepackage{hyperref}
\usepackage{relsize}
\usepackage{blkarray}
\usepackage{tikz}
\usetikzlibrary{decorations.pathreplacing}
\usepackage{tabstackengine}

\stackMath

\usepackage{subfigure}
\usepackage[table]{xcolor}

\usepackage{algpseudocode}

\usepackage{bbm}

\theoremstyle{plain}
\usepackage{amsthm}
\makeatletter
\newcommand{\newreptheorem}[2]{\newtheorem*{rep@#1}{\rep@title}\newenvironment{rep#1}[1]{\def\rep@title{#2 \ref*{##1}}\begin{rep@#1}}{\end{rep@#1}}}
\makeatother

\newtheorem{theorem}{Theorem}
\newtheorem*{theorem-non}{Theorem}
\newtheorem*{non-lemma}{Lemma}
\newtheorem{lemma}[theorem]{Lemma}
\newreptheorem{lemma}{Lemma}

\theoremstyle{definition}

\newcommand{\varsubsetequp}{\mathrel{\rotatebox[origin=c]{90}{$\subseteq$}}}

\newcommand{\Finf}{\mathbb{F}_q^{\mathbb{Z}_+}}

\DeclareMathOperator{\RowS}{RowSpace}

\DeclareMathOperator{\inv}{inv}

\DeclareMathOperator{\id}{id}

\begin{document}

\title{Persistence diagrams of random triangular matrices over finite fields}
\author{Andr\'as M\'esz\'aros}
%}}
\date{}
\affil{HUN-REN Alfr\'ed R\'enyi Institute of Mathematics,\\ Budapest, Hungary}
\maketitle
\begin{abstract}
Let us consider a random infinite lower triangular matrix, where the entries on and below the diagonal are i.i.d. uniform random elements of a fixed finite field. We investigate the evolution of the span of the first $n$ rows of this matrix as $n$ grows. Many properties of this evolving subspace can be captured with the help of the verbose persistence diagram, which is a standard tool in stochastic topology and topological data analysis. We give an explicit formula for the distribution of the persistence diagram. We prove a law of large numbers for the distribution of lifetimes. We also describe the fluctuations of the persistent Betti numbers.

\end{abstract}

\section{Introduction}

Fix a prime power $q$. Let
\[(L(i,j))_{1\le j\le i,\quad i,j\in\mathbb{Z}_+}\]
be i.i.d. uniform random element of the finite field $\mathbb{F}_q$. We set
\[L(i,j)=0\text{ for all positive integers $1\le i<j$.}\]

We think of $L=(L(i,j))_{i,j\in \mathbb{Z}_+}$ as an infinite random lower triangular matrix. 

Let $L_n=(L_{i,j})_{1\le i,j\le n}$ be the top left $n\times n$ submatrix of $L$. We see that $L_n$ is a uniform random $n\times n$ lower triangular matrix over $\mathbb{F}_q$. If one wants to understand the evolution of the coranks
\[\dim \ker L_n=\dim \frac{\mathbb{F}_q^n}{\RowS(L_n)}\]
as $n$ grows, then it is convenient to identify the sequence of spaces $\left(\RowS(L_n)\right)_n$ with a sequence of growing subspaces inside the space $\mathbb{F}_q^{\mathbb{Z}_+}$ as follows.\footnote{Note that we use the definition of~$\mathbb{F}_q^{\mathbb{Z}_+}$, where the element of $\mathbb{F}_q^{\mathbb{Z}_+}$ are only allowed to have finitely many nonzero components.} Let
\[Z_n=\{v\in \mathbb{F}_q^{\mathbb{Z}_+}\,:\,v_i=0\text{ for all }i>n\}.\]
Given a vector $u\in \mathbb{Z}^n$, let $\overline{u}\in \mathbb{F}_q^{\mathbb{Z}_+}$ be obtained from $u$ by appending it with zeros. Relying on the embedding $u\mapsto \overline{u}$, we can identify $\RowS(L_n)$ with the subspace $B_n$ of $\Finf$ defined by
\[B_n=\{\overline{u}\,:\,u\in \RowS(L_n)\}=\{vL\,:\,v\in Z_n\}.\]

The following containments hold:
\begin{equation}\label{eqcontainment}\begin{matrix}
Z_1&\subseteq&Z_2&\subseteq &Z_3&\subseteq&\cdots\\
\varsubsetequp& &\varsubsetequp& &\varsubsetequp& &\cdots\\
B_1&\subseteq&B_2&\subseteq &B_3&\subseteq&\cdots\\
\end{matrix}
\end{equation}

Relying on these definitions, we have
\[\dim\ker L_n=\dim \frac{Z_n}{B_n}.\]

Roughly speaking, $\dim \ker L_n$ measures the number of vectors in $Z_n\setminus B_n$. Motivated by this, given a nonzero vector $v\in \Finf$, we define its birth time $b(v)$ and death time $d(v)$ as
\begin{align*}
 b(v)&=\min \{i\,:\,v\in Z_i\},\\
 d(v)&=\min \{i\,:\,v\in B_i\}.
\end{align*}
Note that $v\in Z_n\setminus B_n$ if and only if $b(v)\le n<d(v)$. Therefore, heuristically a vector $v$ contributes to $\dim \ker L_n$ if and only $b(v)\le n<d(v)$. More precisely,
\[\dim \ker L_n=\dim \frac{Z_n}{B_n}=\dim\{v\in \Finf\,:\,b(v)\le n\}-\dim\{v\in \Finf\,:\,d(v)\le n\}.\]
More intuitively, $\dim \ker L_n$ “counts” the number of vectors which were born no later than time $n$ and died after time $n$.

The notion of (verbose) \textbf{persistence diagrams} provides us a way to record the birth and death times of vectors. In particular, $\dim\ker L_n=\dim \frac{Z_n}{B_n}$ can be recovered from the (verbose) persistence diagram. Besides the dimensions of the quotient spaces $\frac{Z_n}{B_n}$, the persistence diagram captures the relationship of these spaces. More precisely, the sequence of quotient spaces $\frac{Z_n}{B_n}$ can be equipped with an additional algebraic structure giving us the so-called \textbf{persistence module}. The persistence diagram describes this persistence module by providing a decomposition into direct sums of interval modules. See Section~\ref{SecModule} for more details.

We give the definition of verbose persistence diagrams in Lemma~\ref{lemmadef} below. Although most of the statements of this lemma can be found in the literature, we decided to provide a proof, because the specific algorithmic proof we present here will be also used to analyze the behavior of the verbose persistence diagram. Persistence diagrams are often discussed in the setting of the persistent homology of simplicial complexes growing over time. Among the most natural examples of such growing complexes are the \v{C}ech complexes built over point sets in $\mathbb{R}^m$. The notion of persistence diagrams is more intuitive in this geometric/topological setting. We discuss this setting in Section~\ref{SecCech}. The definitions given in Lemma~\ref{lemmadef} are just straightforward analogues of the notions presented in Section~\ref{SecCech}. 

\begin{lemma}\label{lemmadef}
 With probability $1$, we can find a subset $\mathcal{P}$ of
\[\overline{\Delta}_{\mathbb{Z}}=\{(b,d)\in\mathbb{Z}_+^2\,:\,b\le d\}\]
 and a basis $(v_p)_{p\in \mathcal{P}}$ of $\Finf$ indexed by the points of $\mathcal{P}$ with the following properties (the choice of $\mathcal{P}$ and the basis may depend on $L$):
 \begin{enumerate}[(1)]
 \item \label{lemmadedpart1} For any point $p=(b,d)\in\mathcal{P}$, we have $b(v_p)=b$ and $d(v_p)=d$.
 \newpage
 \item \label{lemmadedpart2} \hfill
 \begin{enumerate}[(a)]
 \item \label{lemmadedpart2a} For all $r,s\in \mathbb{Z}_+$, the vectors $\{v_p\,:\,p\in \mathcal{P}\cap ([1,r]\times [1,s])\}$ form a basis of $Z_r\cap B_s$.
 \item \label{lemmadedpart2b} For all $r\in \mathbb{Z}_+$, the vectors $\{v_p\,:\,p\in \mathcal{P}\cap ([1,r]\times [1,\infty))\}$ form a basis of $Z_r$.
 \end{enumerate}
 \item \label{lemmadedpart3} \hfill
 \begin{enumerate}[(a)]
 \item \label{lemmadedpart3a} For all $r,s\in \mathbb{Z}_+$, we have
\[\dim Z_r\cap B_s=|\mathcal{P}\cap ([1,r]\times [1,s])|.\]
 \item \label{lemmadedpart3b} For all $r\in \mathbb{Z}_+$, we have 
\[r=\dim Z_r=|\mathcal{P}\cap ([1,r]\times [1,\infty))|.\]
 \item \label{lemmadedpart3c} For all $r,s\in \mathbb{Z}_+$, we have
\[\dim \frac{Z_r}{Z_r\cap B_s}=|\mathcal{P}\cap ([1,r]\times (s,\infty))|.\]
 \end{enumerate}
 \item \label{lemmadedpart4}\hfill
 \begin{enumerate}[(a)]
 \item \label{lemmadedpart4a}For all $b\in \mathbb{Z}_+$, there is a unique $d\in \mathbb{Z}_+$ such that $(b,d)\in \mathcal{P}$.
 \item \label{lemmadedpart4b} For all $d\in\mathbb{Z}_+$, there is at most one $b\in \mathbb{Z}_+$ such that $(b,d)\in \mathcal{P}$.
 \end{enumerate}
 \item \label{lemmadedpart5} The choice of $\mathcal{P}$ is unique.
 
 \end{enumerate}
\end{lemma}

We call the set $\mathcal{P}$ the verbose persistence diagram corresponding to $(Z_n,B_n)_{n\ge 1}$. Using the notation
\[\Delta_{\mathbb{Z}}=\{(b,d)\in\mathbb{Z}_+^2\,:\,b<d\},\]
the persistence diagram is defined as $\mathcal{P}\cap \Delta_{\mathbb{Z}}$.

We refer to the quantities
\[\dim \frac{Z_r}{Z_r\cap B_s}=|\mathcal{P}\cap ([1,r]\times (s,\infty))|\]
that appear in part~\eqref{lemmadedpart3c} of Lemma~\ref{lemmadef} as \textbf{persistent Betti numbers} and denote them by~$\beta^{r,s}$. 
As promised earlier, the corank of $L_n$ can be recovered from the verbose persistence diagram~$\mathcal{P}$, since
\[\dim\ker L_n=\dim\frac{Z_n}{B_n}=\dim\frac{Z_n}{Z_n\cap B_n}=|\mathcal{P}\cap ([1,n]\times (n,\infty))|=\beta^{n,n}.\]

Given $i\in \mathbb{Z}_+$, we define $d_i$ as the unique $j\in \mathbb{Z}_+$ such that $(i,j)\in\mathcal{P}$. The existence and uniqueness of $d_i$ follows from part~\eqref{lemmadedpart4a} of Lemma~\ref{lemmadef}.

Given $j\in \mathbb{Z}_+$, we define
\begin{equation}\label{eqbidef}b_j=\max\{i\in\mathbb{Z}_+\,:\,(i,j)\in\mathcal{P}\}.\end{equation}
Note that it may happen that $b_{j}=-\infty$, when the maximum above is over an empty set. If $b_{j}\neq-\infty$, then $b_j$ is the unique $i\in\mathbb{Z}_+$ such that $(i,j)\in\mathcal{P}$ as part \eqref{lemmadedpart4b} of Lemma~\ref{lemmadef} states.

With the help of these quantities, we can express $\dim\ker L_n=\beta^{n,n}$ in several other ways. Namely,
\begin{align}
\dim\ker L_n&=
|\{i\in \{1,2,\dots,n\}\,:\,d_i>n\}|\nonumber\\
&=n-|\mathcal{P}\cap [1,n]^2|\nonumber\\
&=|\{j\in \{1,2,\dots,n\}\,:\,b_j=-\infty\}|.\label{eqcorankexpr}
\end{align}

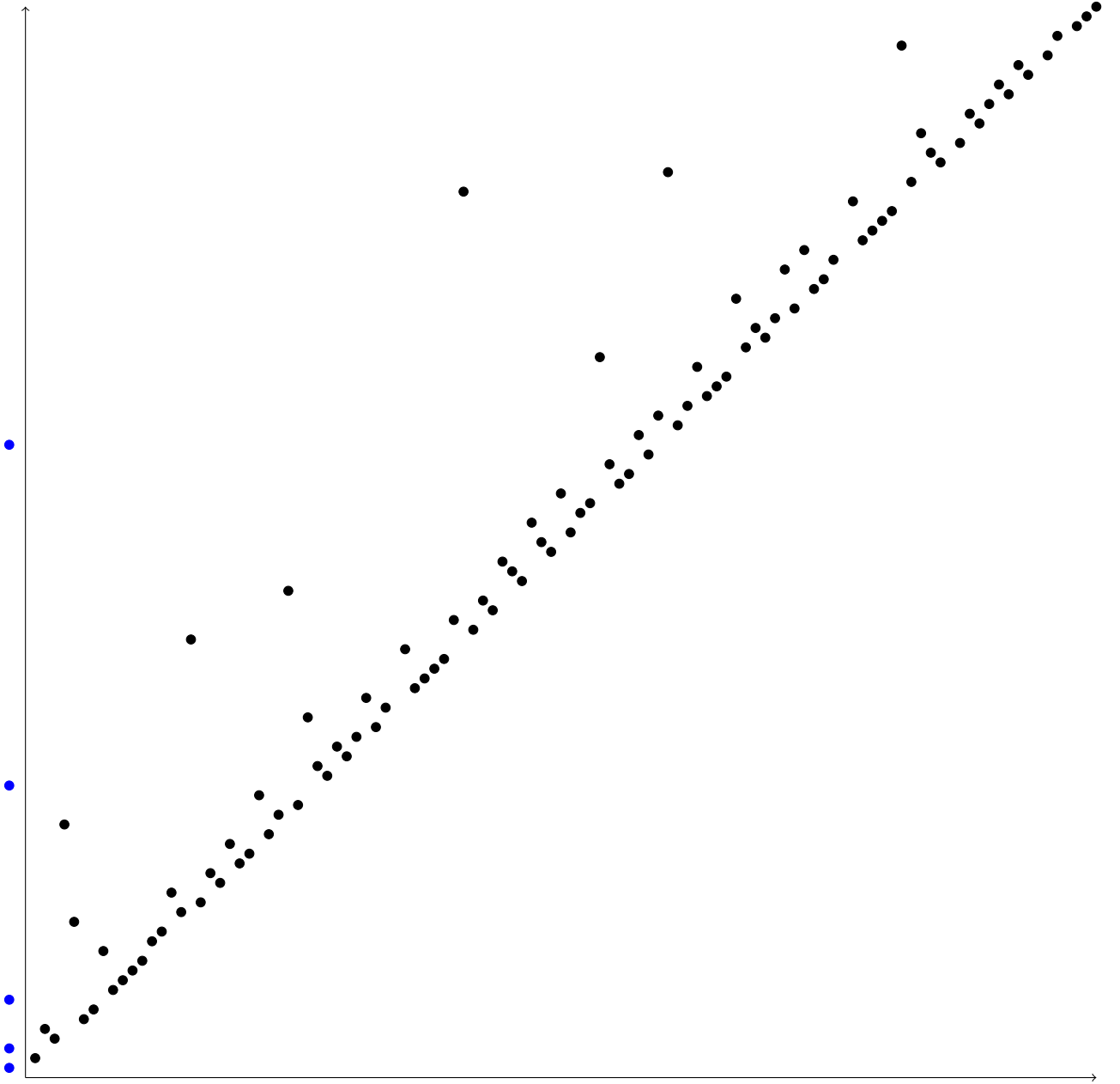
\begin{figure}[h!]

%\centering
\begin{tikzpicture}
\filldraw[blue](-0.25,0.15) circle (2pt);
\filldraw[black](0.15,0.3) circle (2pt);
\filldraw[blue](-0.25,0.45) circle (2pt);
\filldraw[black](0.45,0.6) circle (2pt);
\filldraw[black](0.3,0.75) circle (2pt);
\filldraw[black](0.9,0.9) circle (2pt);
\filldraw[black](1.05,1.05) circle (2pt);
\filldraw[blue](-0.25,1.2) circle (2pt);
\filldraw[black](1.35,1.35) circle (2pt);
\filldraw[black](1.5,1.5) circle (2pt);
\filldraw[black](1.65,1.65) circle (2pt);
\filldraw[black](1.8,1.8) circle (2pt);
\filldraw[black](1.2,1.95) circle (2pt);
\filldraw[black](1.95,2.1) circle (2pt);
\filldraw[black](2.1,2.25) circle (2pt);
\filldraw[black](0.75,2.4) circle (2pt);
\filldraw[black](2.4,2.55) circle (2pt);
\filldraw[black](2.7,2.7) circle (2pt);
\filldraw[black](2.25,2.85) circle (2pt);
\filldraw[black](3,3) circle (2pt);
\filldraw[black](2.85,3.15) circle (2pt);
\filldraw[black](3.3,3.3) circle (2pt);
\filldraw[black](3.45,3.45) circle (2pt);
\filldraw[black](3.15,3.6) circle (2pt);
\filldraw[black](3.75,3.75) circle (2pt);
\filldraw[black](0.6,3.9) circle (2pt);
\filldraw[black](3.9,4.05) circle (2pt);
\filldraw[black](4.2,4.2) circle (2pt);
\filldraw[black](3.6,4.35) circle (2pt);
\filldraw[blue](-0.25,4.5) circle (2pt);
\filldraw[black](4.65,4.65) circle (2pt);
\filldraw[black](4.5,4.8) circle (2pt);
\filldraw[black](4.95,4.95) circle (2pt);
\filldraw[black](4.8,5.1) circle (2pt);
\filldraw[black](5.1,5.25) circle (2pt);
\filldraw[black](5.4,5.4) circle (2pt);
\filldraw[black](4.35,5.55) circle (2pt);
\filldraw[black](5.55,5.7) circle (2pt);
\filldraw[black](5.25,5.85) circle (2pt);
\filldraw[black](6,6) circle (2pt);
\filldraw[black](6.15,6.15) circle (2pt);
\filldraw[black](6.3,6.3) circle (2pt);
\filldraw[black](6.45,6.45) circle (2pt);
\filldraw[black](5.85,6.6) circle (2pt);
\filldraw[black](2.55,6.75) circle (2pt);
\filldraw[black](6.9,6.9) circle (2pt);
\filldraw[black](6.6,7.05) circle (2pt);
\filldraw[black](7.2,7.2) circle (2pt);
\filldraw[black](7.05,7.35) circle (2pt);
\filldraw[black](4.05,7.5) circle (2pt);
\filldraw[black](7.65,7.65) circle (2pt);
\filldraw[black](7.5,7.8) circle (2pt);
\filldraw[black](7.35,7.95) circle (2pt);
\filldraw[black](8.1,8.1) circle (2pt);
\filldraw[black](7.95,8.25) circle (2pt);
\filldraw[black](8.4,8.4) circle (2pt);
\filldraw[black](7.8,8.55) circle (2pt);
\filldraw[black](8.55,8.7) circle (2pt);
\filldraw[black](8.7,8.85) circle (2pt);
\filldraw[black](8.25,9) circle (2pt);
\filldraw[black](9.15,9.15) circle (2pt);
\filldraw[black](9.3,9.3) circle (2pt);
\filldraw[black](9,9.45) circle (2pt);
\filldraw[black](9.6,9.6) circle (2pt);
\filldraw[blue](-0.25,9.75) circle (2pt);
\filldraw[black](9.45,9.9) circle (2pt);
\filldraw[black](10.05,10.05) circle (2pt);
\filldraw[black](9.75,10.2) circle (2pt);
\filldraw[black](10.2,10.35) circle (2pt);
\filldraw[black](10.5,10.5) circle (2pt);
\filldraw[black](10.65,10.65) circle (2pt);
\filldraw[black](10.8,10.8) circle (2pt);
\filldraw[black](10.35,10.95) circle (2pt);
\filldraw[black](8.85,11.1) circle (2pt);
\filldraw[black](11.1,11.25) circle (2pt);
\filldraw[black](11.4,11.4) circle (2pt);
\filldraw[black](11.25,11.55) circle (2pt);
\filldraw[black](11.55,11.7) circle (2pt);
\filldraw[black](11.85,11.85) circle (2pt);
\filldraw[black](10.95,12) circle (2pt);
\filldraw[black](12.15,12.15) circle (2pt);
\filldraw[black](12.3,12.3) circle (2pt);
\filldraw[black](11.7,12.45) circle (2pt);
\filldraw[black](12.45,12.6) circle (2pt);
\filldraw[black](12,12.75) circle (2pt);
\filldraw[black](12.9,12.9) circle (2pt);
\filldraw[black](13.05,13.05) circle (2pt);
\filldraw[black](13.2,13.2) circle (2pt);
\filldraw[black](13.35,13.35) circle (2pt);
\filldraw[black](12.75,13.5) circle (2pt);
\filldraw[black](6.75,13.65) circle (2pt);
\filldraw[black](13.65,13.8) circle (2pt);
\filldraw[black](9.9,13.95) circle (2pt);
\filldraw[black](14.1,14.1) circle (2pt);
\filldraw[black](13.95,14.25) circle (2pt);
\filldraw[black](14.4,14.4) circle (2pt);
\filldraw[black](13.8,14.55) circle (2pt);
\filldraw[black](14.7,14.7) circle (2pt);
\filldraw[black](14.55,14.85) circle (2pt);
\filldraw[black](14.85,15) circle (2pt);
\filldraw[black](15.15,15.15) circle (2pt);
\filldraw[black](15,15.3) circle (2pt);
\filldraw[black](15.45,15.45) circle (2pt);
\filldraw[black](15.3,15.6) circle (2pt);
\filldraw[black](15.75,15.75) circle (2pt);
\filldraw[black](13.5,15.9) circle (2pt);
\filldraw[black](15.9,16.05) circle (2pt);
\filldraw[black](16.2,16.2) circle (2pt);
\filldraw[black](16.35,16.35) circle (2pt);
\filldraw[black](16.5,16.5) circle (2pt);
\draw[->,black, thin] (0,0) -- (0,16.5);
\draw[->,black, thin] (0,0) -- (16.5,0);
\end{tikzpicture}

\caption{A sample of the verbose persistence diagram with the choice of~$q=2$. The black points correspond to points of the verbose persistence diagram with death time at most 110. The blue points left to the death axis correspond to times where the persistence diagram does not contain any points with that given death time, in other words, to times where the corank of $L_n$ grows, see \eqref{eqcorankexpr}. Note that in agreement with Theorem~\ref{thmlifetime}, most points are found near the diagonal line.}
\end{figure}

Theorem~\ref{thmexplicit} below provides an explicit description of the distribution of $\mathcal{P}\cap [1,n]^2$. Before we state this result, we need a few more definitions.

We say that $S\subset \overline{\Delta}_{\mathbb{Z}}\cap [1,n]^2$ is admissible if
\begin{itemize}
 \item For all $b\in \{1,2,\dots,n\}$, we have at most one $d$, such that $(b,d)\in S$.
 \item For all $d\in \{1,2,\dots,n\}$, we have at most one $b$, such that $(b,d)\in S$.
\end{itemize}

Given an admissible set $S\subset \overline{\Delta}_{\mathbb{Z}}\cap [1,n]^2$ and $i\in\{1,2,\dots,n\}$, we define
\[b_i(S)=\begin{cases} \quad b&\text{ if $b$ is the unique number such that $(b,i)\in S$,}\\
-\infty&\text{ if there is not any $b$ such that $(b,i)\in S$.}
\end{cases}.\]

The inversion number of an admissible set $S\subset \overline{\Delta}_{\mathbb{Z}}\cap [1,n]^2$ is defined as 
\[\inv(S)=|\{(i,j)\in \{1,2,\dots,n\}^2\,:\,i<j, b_i(S)>b_j(S)\}|.\]

Finally, we define
\[\Sigma(S)=\sum_{(b,d)\in S} b.\]

%Note that $b>-\infty$ for all $b\in\mathbb{Z}_+$, and $-\infty\not>-\infty$.

\begin{theorem}\label{thmexplicit}\hfill
\begin{enumerate}[(a)]
 \item With probability $1$, the set $\mathcal{P}\cap [1,n]^2$ is admissible.
 \item For any admissible set $S\subset \overline{\Delta}_{\mathbb{Z}}\cap [1,n]^2$, we have
\[\mathbb{P}(\mathcal{P}\cap [1,n]^2=S)=\left(\frac{q-1}{q}\right)^{|S|}q^{\Sigma(S)+\inv(S)-{{n+1}\choose{2}}}.\]
\end{enumerate}

\end{theorem}

Given a point $(b,d)\in \mathcal{P}$, the lifetime of this point is defined as $d-b$. Our next theorem provides a law of large numbers for the distribution of lifetimes.

\begin{theorem}\label{thmlifetime}
Let\[X=\sum_{i=1}^{G-1} Y_i,\]
where $G$ is a shifted geometric random variable\footnote{A $\mathbb{Z}_+$-valued random variable $Z$ is a shifted geometric random variable with success probability $p$, if $\mathbb{P}(Z=k)=p(1-p)^{k-1}$ for all $k\in\mathbb{Z}_+$.} with success probability $\frac{q-1}q$, $Y_i$ is a shifted geometric random variable with success probability $q^{-i}$, and $G,Y_1,Y_2,Y_3,\dots$ are all independent.

For almost all choices of $L$, we have that for all $k\in \mathbb{Z}_{\ge 0}$,
\[\lim_{n\to\infty} \frac{|\{(b,d)\in \mathcal{P}\cap [1,n]^2\,:\,d-b=k\}|}{n}=\mathbb{P}(X=k).\]

\end{theorem}

In Theorem~\ref{thmfluc} below, we describe the limiting fluctuations of the Betti numbers $\beta^{r_n,r_n+t_n}$, where $r_n\to \infty$ and $t_n\ge 0$. We will distinguish four phases depending on how fast $t_n$ grows.

Before doing this, we need a few definitions.

For $u\ge 0$, let $J_u$ be a $\mathbb{Z}_{\ge 0}$-valued random variable such that
\[\mathbb{P}(J_u=k)=q^{-k(k+u)} \prod_{i=1}^k (1-q^{-i})^{-1}\prod_{i=1}^{k+u} (1-q^{-i})^{-1}\prod_{i=1}^\infty (1-q^{-i}).\]

The distribution above arises as the universal limiting distribution of the corank of $(n+u)\times n$ matrices with independent entries~\cite{wood2019random}. 

Given $\chi>0$, let $\mathcal{L}_{1,q^{-1},\chi}$ be a $\mathbb{Z}$-valued random variable such that for all $x\in \mathbb{Z}$, we have

\[\mathbb{P}(\mathcal{L}_{1,q^{-1},\chi}=x)=\frac{1}{\prod_{i=1}^{\infty} (1-q^{-i})}\sum_{m=0}^\infty \exp(-\chi q^{m-x})\frac{(-1)^m q^{-{{m}\choose{2}}}}{\prod_{j=1}^m (1-q^{-j})}.\]

Van Peski~\cite{van2026rank} used the random variables $\mathcal{L}_{1,q^{-1},\chi}$ to describe the limiting fluctuations of $\dim\ker L_n=\beta^{n,n}$, see part \eqref{fluc1} of Theorem~\ref{thmfluc} and Lemma~\ref{LemmaRVP}. The author extended these results to lower triangular matrices where the entries on and below the diagonal are still i.i.d. but no longer uniform~\cite{meszaros2025fluctuations}. The random variables $\mathcal{L}_{1,q^{-1},\chi}$ appeared earlier in the work of Van Peski~\cite{vanp1,vanp2,vanp3} on the corank of matrix products, or more generally, on the cokernels of matrix products of Haar uniform matrices over the $p$-adic integers. Nguyen and Van Peski~\cite{nguyen2024rank} proved that this behavior of matrix products is universal, that is, the cokernels of products of matrices with i.i.d. entries show the same limiting behavior regardless of the distribution of the entries. The author showed that besides matrix products and lower triangular matrices this universality class contains a large class of block lower triangular matrices~\cite{meszaros2025universal,meszaros2025rank}.

Finally, we also need the technical assumption that the fractional parts $\{-\log_q(r_n)\}$ converge to $\zeta$. Note that we can always ensure this condition by passing to a subsequence. Let $\mathcal{L}=\mathcal{L}_{1,q^{-1},q^{-1-\zeta}}$. 

For $x\in \mathbb{R}$, let $\lfloor x \rceil=\lfloor x+0.5 \rfloor$ be the integer closest to $x$.

\begin{theorem}\hfill\label{thmfluc}
\begin{enumerate}[(1)]
 \item \label{fluc1} Assume that $t_n=t$ for a nonnegative constant $t$. Let $D$ be a random variable independent from $\mathcal{L}$ such that $D$ has the same distribution as $\dim\ker L_t$. Then, as $n\to\infty$,
\[\beta^{r_n,r_n+t}-\lfloor \log_q(r_n)+\zeta\rceil\text{ converges in distribution to }\mathcal{L}-D.\]
 \item \label{fluc2} Assume that $t_n\to\infty$ and $\log_q(t_n)-\log_q(r_n)\to-\infty$. Moreover, assume that fractional parts $\{-\log_q(t_n)\}$ converge to $\zeta'$. Let $\mathcal{L}'=\mathcal{L}_{1,q^{-1},q^{-1-\zeta'}}$ be independent from $\mathcal{L}$. Then, as $n\to\infty$,
\[\beta^{r_n,r_n+t_n}-(\lfloor \log_q(r_n)+\zeta\rceil-\lfloor \log_q(t_n)+\zeta'\rceil)\text{ converges in distribution to }\mathcal{L}-\mathcal{L}'.\]
 \item \label{fluc3} Assume that $\log_q(t_n)-\log_q(r_n)\to\gamma \in \mathbb{R}$. Let $\zeta'=\{\zeta-\gamma\}$. Let $\mathcal{L}'=\mathcal{L}_{1,q^{-1},q^{-1-\zeta'}}$ be independent from $\mathcal{L}$. Let $\mathcal{J}=\gamma+\zeta'-\zeta+\mathcal{L}'-\mathcal{L}$. Then
\[\beta^{r_n,r_n+t_n}\text{ converges in distribution to }\max(0,-\mathcal{J})+J_{|\mathcal{J}|},\]
 where $J_0,J_1,\dots$ are independent from $\mathcal{J}$.
 \item \label{fluc4} Assume that $\log_q(t_n)-\log_q(r_n)\to\infty$, then
\[\lim_{n\to\infty}\mathbb{P}(\beta^{r_n,r_n+t_n}=0)=1.\]
\end{enumerate}
\end{theorem}

In the past decade, the rank of random matrices over finite fields, and more generally the cokernels of random matrices, have been studied extensively. A primary motivation for this line of research is the conjecture that the statistics of the class groups of quadratic number fields behave like the cokernels of Haar-uniform random matrices over the $p$-adic integers~\cite{cohen2006heuristics,venkatesh2010statistics}. The limiting distribution of the cokernels of Haar-uniform random square matrices over the $p$-adic integers is given by the Cohen-Lenstra distribution~\cite{friedman1989distribution}. Much of the research in this area has focused on proving universality results: for a large class of random matrices, the limiting distribution of their cokernels is given by the Cohen-Lenstra distribution or a variant thereof~\cite{wood2019random,recent1,recent2,recent3,recent4,recent5,recent6,recent7,recent8,recent9,recent10,meszaros2020distribution,meszaros2024phase}. The majority of these results were obtained using Wood's powerful moment method~\cite{wood2017distribution} and its generalizations~\cite{sawin2022moment,nguyen2024rank}.
In recent years, a growing number of results have instead adopted a dynamical perspective. In this framework, one considers a randomly evolving sequence of matrices and describes the evolution of their cokernels. This viewpoint is valuable even when the primary interest lies in the single-time marginals of the process. This perspective is most prominently present in the work of Van Peski on the time evolution of the cokernels of matrix products~\cite{vanp1,vanp2,vanp3}. See also \cite{lvov2024random,meszaros2025rank,shen2026quantative} for other results in $p$-adic matrix theory where this dynamical perspective and Markov-chain methods play a key role.

Another rich source of random abelian groups arises from the homology groups of random simplicial complexes. If these complexes evolve over time such as the \v{C}ech complexes considered in Section~\ref{SecCech}, then they give a random process of abelian groups. Persistent homology provides a framework for understanding this evolution.

Since homology groups can be viewed as cokernels, the study of the homology of random simplicial complexes could be seen as a part of $p$-adic random matrix theory. Despite this connection, there has been little transfer of ideas between these two fields~\cite{kahle2020cohen,lee2025distribution,meszaros2023cohen,meszaros20242,newman2023abelian}. The goal of this paper is to bridge this gap by applying the notion of persistent homology (a tool well-established in stochastic topology and topological data analysis) to the setting of lower triangular matrices, thereby providing a new perspective on this this evolving area of $p$-adic random matrix theory. 

We think that persistence diagrams can be useful for describing the time evolution of the corank of other random matrix models that evolve over time.

\bigskip

\textbf{Acknowledgments.} The author was supported by the Marie Sk\l{}odowska-Curie Postdoctoral Fellowship "RaCoCoLe".

\section{More background on persistence diagrams}
\subsection{Persistent homology of growing simplicial complexes}\label{SecCech}

Let $S$ be a finite set of points in $\mathbb{R}^m$. We would like to make a formal sense of the question: What is the shape of the set $S$? One possibility is to choose a radius $r\ge 0$ and consider the set $S_r\subset \mathbb{R}^m$ defined as the union of the $r$-balls centered around the points of $S$, then try to understand the homological properties of this set $S_r$. Of course, we should choose the radius~$r$ carefully. If $r$ is too small, then $S_r$ is just a disjoint union of balls. If $r$ is too large, then $S_r$ is contractible. To overcome the problem of the choice of the radius $r$, we can just consider all the possible choices of $r$, and try to understand the evolution of $S_r$ as $r$ grows. Given any $1\le d\le m-1$, as we increase $r$, new $d$-dimensional holes can be born in $S_r$, while other holes might die, when they are filled in. We can use tools from the theory of \textbf{persistent homology} to capture the changes of the homological features of~$S_r$. In particular, \textbf{persistence diagrams} can be used to record the birth and death times of the $d$-dimensional holes (homology classes). Instead of giving a formal construction, we provide an example in Figure~\ref{Figure1}. A few paragraphs below, we provide a definition in terms of persistent Betti numbers, but we think that this definition is less illuminating than Figure~\ref{Figure1}.

\def\innercolor{red!30}
\def\outercolor{red!18}
\def\drawcoloe{red!20}
\def\points{
\begin{scope}[blend group=multiply]
\draw[even odd rule,inner color=\innercolor,outer color=\outercolor, draw=\drawcoloe](0,0) circle (\sugar);
\draw[even odd rule,inner color=\innercolor,outer color=\outercolor, draw=\drawcoloe](-0.5,-0.6) circle (\sugar);
\draw[even odd rule,inner color=\innercolor,outer color=\outercolor, draw=\drawcoloe](-0.6,0.6) circle (\sugar);
\draw[even odd rule,inner color=\innercolor,outer color=\outercolor, draw=\drawcoloe](-1.3,0.1) circle (\sugar);%1.3
\draw[even odd rule,inner color=\innercolor,outer color=\outercolor, draw=\drawcoloe](-0.1,1.7) circle (\sugar);
\draw[even odd rule,inner color=\innercolor,outer color=\outercolor, draw=\drawcoloe](0.9,1.8) circle (\sugar);
\draw[even odd rule,inner color=\innercolor,outer color=\outercolor, draw=\drawcoloe](2.1,0.9) circle (\sugar);
\draw[even odd rule,inner color=\innercolor,outer color=\outercolor, draw=\drawcoloe](2,0) circle (\sugar);
\draw[even odd rule,inner color=\innercolor,outer color=\outercolor, draw=\drawcoloe](1,-0.3) circle (\sugar);
\draw[even odd rule,inner color=\innercolor,outer color=\outercolor, draw=\drawcoloe](2.6,-1.6) circle (\sugar);
\draw[even odd rule,inner color=\innercolor,outer color=\outercolor, draw=\drawcoloe](-0.8,-1.9) circle (\sugar);
\draw[even odd rule,inner color=\innercolor,outer color=\outercolor, draw=\drawcoloe](0.6,-2.95) circle (\sugar);
\draw[even odd rule,inner color=\innercolor,outer color=\outercolor, draw=\drawcoloe](1.4,-2.85) circle (\sugar);
\end{scope}
\filldraw[black](0,0) circle (2pt);
\filldraw[black](-0.5,-0.6) circle (2pt);
\filldraw[black](-0.6,0.6) circle (2pt);
\filldraw[black](-1.3,0.1) circle (2pt);
\filldraw[black](-0.1,1.7) circle (2pt);
\filldraw[black](0.9,1.8) circle (2pt);
\filldraw[black](2.1,0.9) circle (2pt);
\filldraw[black](2,0) circle (2pt);
\filldraw[black](1,-0.3) circle (2pt);
\filldraw[black](2.6,-1.6) circle (2pt);
\filldraw[black](-0.8,-1.9) circle (2pt);
\filldraw[black](0.6,-2.95) circle (2pt);
\filldraw[black](1.4,-2.85) circle (2pt);

\draw[white, very thick] (-2.8,-4.38) rectangle (4.2,3.5);
}

\begin{figure*}
 \centering
 \begin{subfigure}[\hspace{-8pt}Only disjoint balls.]{%
 \def\sugar{0.3}
\resizebox{0.22\linewidth}{!}{
\begin{tikzpicture}
\points
\end{tikzpicture}}
 }\end{subfigure}\hspace{0cm}
\begin{subfigure}[\hspace{-8pt}The purple cycle is born.]{%
 \def\sugar{0.514}%0.515
\resizebox{0.22\linewidth}{!}{
\begin{tikzpicture}
\points
\draw[purple, ultra thick] (0,0) -- (-0.5,-0.6) -- (-1.3,0.1) -- (-0.6,0.6)-- cycle;
\end{tikzpicture}}
 }
\end{subfigure}\hspace{0cm}
\begin{subfigure}[\hspace{-8pt}The purple cycle dies.]{%
 \def\sugar{0.6}%{0.54}
\resizebox{0.22\linewidth}{!}{
\begin{tikzpicture}
\points
\end{tikzpicture}}
 }\end{subfigure}\hspace{0cm}
\begin{subfigure}[\hspace{-8pt}The blue cycle is born.]{%
 \def\sugar{0.74}
\resizebox{0.22\linewidth}{!}{
\begin{tikzpicture}
\points
\draw[blue, ultra thick] (0,0) -- (-0.6,0.6)-- (-0.1,1.7)--(0.9,1.8)--(2.1,0.9)--(2,0)--(1,-0.3) -- cycle;
\end{tikzpicture}}
 }\end{subfigure}
\\
 \begin{subfigure}[\hspace{-8pt}The green cycle is born.]{%
 \def\sugar{0.855}
\resizebox{0.22\linewidth}{!}{
\begin{tikzpicture}
\points
\draw[blue, ultra thick] (0,0) -- (-0.6,0.6)-- (-0.1,1.7)--(0.9,1.8)--(2.1,0.9)--(2,0)--(1,-0.3) -- cycle;

\draw[green, ultra thick] (0,0) -- (-0.5,-0.6)--(-0.8,-1.9)-- (0.6,-2.95)--(1.4,-2.85)--(2.6,-1.6)--(2,0)--(1,-0.3) -- cycle;
\end{tikzpicture}}
 }\end{subfigure}\hspace{0cm}
\begin{subfigure}[\hspace{-8pt}The blue cycle dies.]{%
 \def\sugar{1.05}
\resizebox{0.22\linewidth}{!}{
\begin{tikzpicture}
\points
\draw[green, ultra thick] (0,0) -- (-0.5,-0.6)--(-0.8,-1.9)-- (0.6,-2.95)--(1.4,-2.85)--(2.6,-1.6)--(2,0)--(1,-0.3) -- cycle;
\end{tikzpicture}}
 }\end{subfigure}\hspace{0cm}
\begin{subfigure}[\hspace{-8pt}The green cycle dies.]{%
 \def\sugar{1.35}
\resizebox{0.22\linewidth}{!}{
\begin{tikzpicture}
\points
\end{tikzpicture}}
 }\end{subfigure}\hspace{0cm}
\begin{subfigure}[\hspace{-8pt}The persistence diagram]{%
 \def\sugar{0.74}
\resizebox{0.22\linewidth}{!}{
\begin{tikzpicture}

\filldraw[purple](0.515,0.6) circle (1.2pt);
\filldraw[blue](0.74,1.05) circle (1.2pt);
\filldraw[green](0.855,1.35) circle (1.2pt);
\draw[gray, thin] (0,0) -- (1.5,1.5);
\draw[->,black, thin] (0,0) -- (0,1.5) node[anchor=south]{\scalebox{.4}{death}};
\draw[->,black, thin] (0,0) -- (1.5,0) node[anchor=north]{\scalebox{.4}{birth}};; 

\end{tikzpicture}}
 }\end{subfigure}

 \caption{The persistence diagram above displays the birth and death times of the $1$-dimensional homology classes.}
 \label{Figure1}
 \end{figure*}
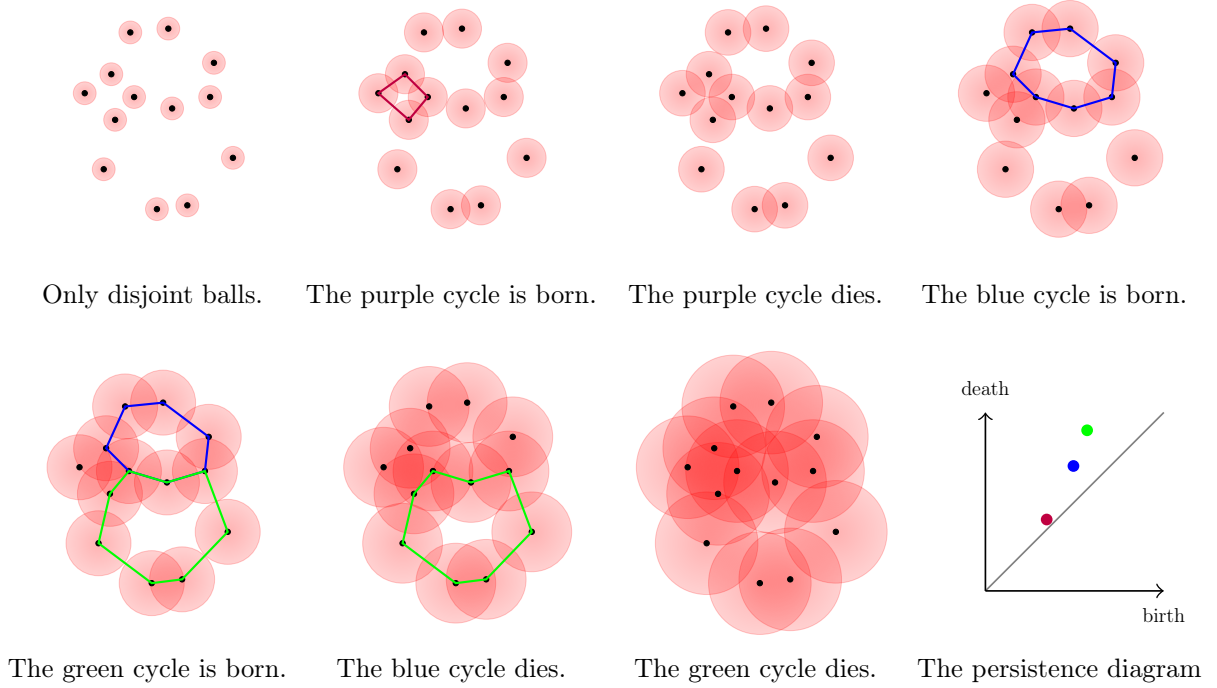

Persistent homology was introduced by Edelsbrunner, Letscher and Zomorodian~\cite{edelsbrunner2002topological}. Their definition was motivated by the results of Delfinado, Edelsbrunner, Kirkpatrick, M\"ucke and Seide on alpha shapes and related computational problems~\cite{alpha1,alpha2,alpha3}. Similar ideas appeared earlier in the work of Robins~\cite{robins1999towards} on finite point-set approximations of compact metric spaces, and in the theory of size functions by Frosini~\cite{frosini1990distance}. Persistent homology is also related to spectral sequences originally introduced by Leray~\cite{leray1946structure}. To highlight a connection to Morse-theory, we mention the following observation: For a smooth function $f:M\to \mathbb{R}$ on a sufficiently nice space $M$, consider the superlevel sets $A_u=f^{-1}([u,\infty))$. As we increase $u$, changes in the homology of the superlevel sets $A_u$ correspond to critical points of $f$. See also \cite{adler2010persistent,weinberger2011persistent}.

Returning to the time evolution of the sets $S_r$, \textbf{\v{C}ech complexes} provide useful combinatorial tools to understand the homology of~$S_r$. The \v{C}ech complex $C(r)$ is defined as the simplicial complex on the vertex set~$S$, where a subset $\sigma$ of $S$ is present in the complex~$C(r)$ if and only if the $r$-balls centered around the points of $\sigma$ have a nonempty intersection. Note that $(C(r))_{r\ge 0}$ is a process of growing simplicial complexes. By the nerve lemma, the \v{C}ech complex~$C(r)$ is homotopy equivalent to~$S_r$. Thus, the homology of $C(r)$ is the same as the homology of $S_r$. 

Fix a dimension $i$ and a field $F$. Given $r,s\in \mathbb{R}_+$, the $i$-dimensional persistent Betti number $\beta_i^{r,s}$ is defined as
\[\beta_i^{r,s}=\dim\frac{Z_i(C(r),F)}{Z_i(C(r),F)\cap B_i(C(s),F)}.\]

Heuristically, $\beta_i^{r,s}$ counts the number of $i$-dimensional holes that were born before time~$r$ and still present after time $s$. Usually, we are interested in the case $r\le s$, but the definition makes sense even for $r>s$.

The i-dimensional verbose persistence diagram of $(C(r))_{r\ge 0}$ is defined to be the unique multiset\footnote{Note that there are other settings when we should allow death times to be $\infty$, but for the purposes of this paper, we can ignore this detail.}\[\mathcal{P}\subset \overline{\Delta}=\{(b,d)\in \mathbb{R}_{\ge 0}\,:\,b\le d\}\] such that 
\begin{equation}\label{fundamental}\beta_i^{r,s}=|\mathcal{P}\cap ([0,r]\times(s,\infty))|\text{ for all }r,s\in\mathbb{R}_+.\end{equation}
The persistence diagram is obtained from the verbose persistence diagram by discarding the points of the form $(x,x)$. Observe that \eqref{fundamental} remains true if we replace the verbose persistence diagram with the persistence diagram under the additional assumption that~$r\le s$. 

One can see that the persistence diagram constructed in Lemma~\ref{lemmadef} is analogous to the construction above. As Figure~\ref{Figure1} suggests, one can also assign a cycle to each point of the persistence diagram. This is also analogous to the system of vectors $(v_p)_{p\in \mathcal{P}}$ discussed in Lemma~\ref{lemmadef}.

See \cite{hiraoka2018limit,bobrowski2022homological,bobrowski2017maximally,bobrowski2023universal,bobrowski2024universality,edelsbrunner2024maximum} for results on the homology of random \v{C}ech complexes, and see~\cite{hiraoka2017minimum,hino2019asymptotic,meszaros2025persistent} for results on the persistent homology of other random simplicial complexes.

\subsection{Persistence modules}\label{SecModule}

A \textbf{persistence module} is a pair $\left((V_t)_{t\in \mathbb{R}_+},(\varrho_{s,t})_{0\le s\le t}\right)$, where $V_t$ is a finite dimensional vector space over some fixed field $F$, and $\varrho_{s,t}:V_s\to V_t$ is a linear map, these maps are required to satisfy that $\varrho_{s,t}\circ\varrho_{r,s}=\varrho_{r,t}$ for all $r\le s\le t$, and $\varrho_{t,t}=\id$ for all $t$. 

The simplest examples of persistence modules are given by interval modules: Given a subinterval $I$ of $\mathbb{R}_+$, the corresponding interval module $\left((V^I_t)_{t\in \mathbb{R}_+},(\varrho^I_{s,t})_{0\le s\le t}\right)$ is defined as follows:
\begin{align*}V_t^I&=\begin{cases} F&\text{ if }t\in I,\\
0&\text{ otherwise,}\end{cases}\\
\varrho^I_{s,t}&=\begin{cases}
 \id&\text{ if }s,t\in I,\\
 0&\text{ otherwise.}
\end{cases}
\end{align*}

Each persistence module can be decomposed as the direct sum of interval modules~\cite{crawley2015decomposition}.

Let $(K(r))_{t\ge 0}$ be a process of growing simplicial complexes such as the \v{C}ech complexes considered in Section~\ref{SecCech}. Fix a dimension $i$. Then for any $s\le t$, one can define a linear map $\varrho_{s,t}:H_i(K(s),F)\to H_i(K(t),F)$ by setting $\varrho_{s,t}(z+B_i(K(s),F))=z+B_i(K(t),F)$. Then $\left(\left(H_i(K(t),F)\right)_{0\le t},\left(\varrho_{s,t}\right)_{0\le s\le t}\right)$ is a persistence module. Decompose this module as the direct sum of interval module, and let $\mathcal{I}$ be the multiset of the intervals that appear in this decomposition. Then $(\inf I,\sup I)_{I\in\mathcal{I}}$ gives the persistence diagram of the process $(K(t))_{t\ge 0}$. Thus, we obtain another perspective on persistence diagrams.

One can also consider persistence modules indexed by $\mathbb{Z}_+$, and everything above remains true. Given a sequence $(Z_t,B_t)_{t\in \mathbb{Z}_+}$ of pairs of finite dimensional subspaces of some vector space satisfying the containment requirements in \eqref{eqcontainment}, one can consider the persistence module $((Z_t/B_t)_{t\in\mathbb{Z}_+},(\varrho_{s,t})_{1\le s\le t})$, where similarly to the case above we define $\varrho_{s,t}(z+B_s)=z+B_t$. In particular, the persistence diagram given in Lemma~\ref{lemmadef} corresponds to the decomposition of the corresponding persistence module as a direct sum of interval modules.

\section{The proof of Lemma~\ref{lemmadef} and an algorithm for sampling the persistence diagram}

\subsection{The first steps towards constructing the persistence diagram}\label{SecFirstSteps}

% \begin{lemma}\label{lemmaperind}
% For any nonnegative integer $n$, there exist $\mathcal{P}_n\subset [1,n]^2\cap \overline{\Delta}$ and a basis $(v_p)_{p\in \mathcal{P}_n}$ of $B_n$ such that 
% \begin{itemize}
% \item for all $r,s\in \mathbb{Z}_+$ where $s\le n$, the vectors $\{v_p\,:\,p\in\mathcal{P}_n\cap([1,r]\times [1,s]) \}$ form a basis of $Z_r\cap B_s$;
% \item for all $p=(b,d)\in\mathcal{P}_n$, we have $b(v_p)=b$ and $d(v_p)=d$.
% \end{itemize}

% Moreover, if $\dim B_{n}=\dim B_{n-1}$, we can choose $\mathcal{P}_n=\mathcal{P}_{n-1}$ and use the same basis. If $\dim B_{n}=\dim B_{n-1}+1$, $\mathcal{P}_n$ can be obtained from $\mathcal{P}_{n-1}$ by adding a point of the form $p=(b_n,n)$ and extending the basis by a single vector $v_p$.

% \end{lemma}

% \begin{proof}

For $r\in \mathbb{Z}_{\ge 0}$, let $\Pi_r:\Finf\to\Finf$ be the projection defined by
\[(\Pi_r v)_i=\begin{cases}
v_i&\text{ if $i>r$,}\\
0&\text{ if $i\le r$.}
\end{cases}\]

% For $r,s\in \mathbb{Z}_{\ge 0}$, let us define the random matrix $L_{(r,s]}=\left(L_{(r,s]}(i,j)\right)_{i,j\in\mathbb{Z}_+}$ by
%\[L_{(r,s]}(i,j)=\begin{cases}
% 0&\text{ if $i>s$ or $j\le r$}\\
% L(i,j)&\text{ otherwise.}
% \end{cases}\]

% Note that if $r\ge s$, then $L_{(r,s]}$ is the all zero matrix. Let
%\[R(r,s)=\rang(L_{(r,s]})=\dim \RowS(L_{(r,s]}).\]
% Observe that if $r\ge s$, then $R(r,s)=0$. Moreover, $R(0,n)=\dim B_n$. 

Applying the rank-nullity theorem to the map $\Pi_r$ restricted to $B_s$, we get that

\begin{equation}\label{ZrcapBs}
 \dim Z_r\cap B_s=\dim\ker (\Pi_r\restriction B_s)=\dim B_s-\dim \Pi_r B_s.
\end{equation}

Next, for $r,s\in \mathbb{Z}_+$, we define
\begin{equation}\label{eqAdef}A(r,s)=\dim\Pi_{r-1}B_s-\dim\Pi_r B_s.\end{equation}
Since $\dim\Pi_r B_s=0$ for all $r\ge s$, we see that
\[A(r,s)=0\text{ whenever }r>s.\]
Due to this, when we are discussing algorithms to generate $A$, we often omit the step that we should set $A(r,s)=0$ for all $r>s$, and we only concentrate on the entries where $r\le s$.

By applying the rank-nullity theorem for the map $\Pi_r$ restricted to $\Pi_{r-1}B_s$, we see that
\[A(r,s)=\dim ((\Pi_{r-1} B_s)\cap \ker \Pi_r).\] 
Combining this with the fact that
\[(\Pi_{r-1} B_s)\cap \ker \Pi_r\subset \{v\in \Finf\,:\, v_i=0\text{ for all }i\neq r\},\]
we see that
\[A(r,s)\in \{0,1\}.\]

Let $z_n$ be the $n$-th row of $L$. Since $\Pi_r B_n$ is generated by the elements of $\Pi_r B_{n-1}$ and~$\Pi_r z_n$, we see that
\begin{equation}\label{eqdimPirB0}\dim\Pi_r B_n=\begin{cases}
 \dim \Pi_r B_{n-1}&\text{ if $\Pi_r z_n\in \Pi_r B_{n-1}$,}\\
 \dim \Pi_r B_{n-1}+1&\text{ if $\Pi_r z_n\notin \Pi_r B_{n-1}$.}
\end{cases}\end{equation}

Let \begin{equation}\label{bidef2}b_n=\max\{i\in\mathbb{Z}_+\,:\,\Pi_{i-1}z_n\notin \Pi_{i-1}B_{n-1}\}.\end{equation}
We will see later that this definition is consistent with the one given in~\eqref{eqbidef}.

By the definition of $b_n$, if $r\ge b_n$, then $\Pi_r z_n\in \Pi_r B_{n-1}$. 

We claim that if $r<b_n$, then $\Pi_r z_n\notin \Pi_r B_{n-1}$. We prove this statement by contradiction, so assume that $r<b_n$ and $\Pi_r z_n\in \Pi_r B_{n-1}$. Then $\Pi_{b_n-1}\Pi_r z_n\in \Pi_{b_n-1}\Pi_r B_{n-1}$. Observing that~$\Pi_{b_n-1}\Pi_r=\Pi_{b_n-1}$, it follows that $\Pi_{b_n-1}z_n\in \Pi_{b_n-1}B_{n-1}$, which is a contradiction. Combining these with~\eqref{eqdimPirB0}, we see that

\begin{equation}\label{eqdimPirB}\dim\Pi_r B_n=\begin{cases}
 \dim \Pi_r B_{n-1}&\text{ if $r\ge b_n$,}\\
 \dim \Pi_r B_{n-1}+1&\text{ if $r< b_n$.}
\end{cases}\end{equation}

Thus, if $r<b_n$, then
\[A(r,n)=\dim\Pi_{r-1}B_n-\dim\Pi_r B_n=(\dim\Pi_{r-1}B_{n-1}+1)-(\dim\Pi_r B_{n-1}+1)=A(r,n-1).\]
Furthermore, if $r>b_n$, then
\[A(r,n)=\dim\Pi_{r-1}B_n-\dim\Pi_r B_n=\dim\Pi_{r-1}B_{n-1}-\dim\Pi_r B_{n-1}=A(r,n-1).\]
Finally, assuming that $b_n\neq-\infty$, we have
\[A(b_n,n)=\dim\Pi_{b_n-1}B_n-\dim\Pi_{b_n} B_n=(\dim\Pi_{b_n-1}B_{n-1}+1)-\dim\Pi_{b_n} B_{n-1}=A(b_n,n-1)+1.\]
Since $A(b_n,n),A(b_n,n-1)\in \{0,1\}$, this is only possible if
\[A(b_n,n-1)=0\quad\text{ and }\quad A(b_n,n)=1.\]

We define $\mathcal{P}$ as
\[\mathcal{P}=\{(b_n,n)\,:\,n\in\mathbb{Z}_+\text{ and }b_n\neq -\infty\}.\]

From now on, $\mathcal{P}$ will mean the set constructed as above. Later we will see that this~$\mathcal{P}$ satisfies all the properties given in Lemma~\ref{lemmadef}, and it is uniquely determined by these properties.

We see that the two definitions of $b_i$ given in \eqref{eqbidef} and \eqref{bidef2} are consistent.

Consider an $n\in \mathbb{Z}_+$ such that $b_n\neq -\infty$. Let $p_n=(b_n,n)$. By the definition of $b_n$, we see that $\Pi_{b_n} z_n\in \Pi_{b_n} B_{n-1}$. Thus, there is a vector $z_n'\in B_{n-1}$ such that $\Pi_{b_n} z_n'=\Pi_{b_n}z_n$. Let\[v_{p_n}=z_n-z_n'.\] \begin{lemma}\label{banddgood}
We have $b(v_{p_n})=b_n$ and $d(v_{p_n})=n$.
\end{lemma}
\begin{proof}
Observe that $v_{p_n}\in\ker \Pi_{b_n}=Z_{b_n}$. It follows from the definition of $b_n$ given in \eqref{bidef2} that $\Pi_{b_n-1} z_n\notin \Pi_{b_n-1} B_{n-1}$. In particular, $\Pi_{b_n-1} z_n\neq \Pi_{b_n-1} z_n'$, so $v_{p_n}\notin \ker \Pi_{b_n-1}=Z_{b_n-1}$. Therefore, $v_{p_n}\in Z_{b_n}\setminus Z_{b_n-1}$, which implies that $b(v_{p_n})=b_n$. Since $z_n\in B_n\setminus B_{n-1}$, we see that $v_{p_n}\in B_n\setminus B_{n-1}$, so $d(v_{p_n})=n$. 
\end{proof}
%Since $v_{p_n}\notin B_{n-1}$, it is obviously linearly independent from $\{v_{p}\,:\,p\in \mathcal{P}_{n-1}\}$. Since $\dim B_n=\dim B_{n-1}+1$, 
\begin{lemma}\label{lemmadef0ind}\hfill
\begin{enumerate}[(a)]
\item\label{lemmdef0parta}For all $n$, the vectors $\{v_{p}\,:\,p\in \mathcal{P}\cap [1,n]^2\}$ form a basis of $B_n$. 
\item\label{lemmdef0partb}For all $r,n\in\mathbb{Z}_+$, the vectors $\{v_{p}\,:\,p\in \mathcal{P}\cap [1,r]\times [1,n]\}$ form a basis of $Z_r\cap B_n$.
\end{enumerate}
\end{lemma}
\begin{proof}
We prove by induction on $n$. We first start with the case $b_n\neq-\infty$. Let $p_n=(b_n,n)$.

By the induction hypothesis, the vectors $\{v_p\,:\, p\in\mathcal{P}\cap [1,n-1]^2\}$ form a basis of~$B_{n-1}$. Combining this with the facts that $v_{p_n}\in B_n\setminus B_{n-1}$ and $\dim B_n=\dim B_{n-1}+1$, we see that the vectors $\{v_{p}\,:\,p\in \mathcal{P}\cap [1,n]^2\}$ form a basis of $B_n$. This finishes the proof of part~\eqref{lemmdef0parta}, we move on to part~\eqref{lemmdef0partb}.

If $r<b_n$, then 
\begin{align*}\dim& Z_r\cap B_n&\\&=\dim B_n-\dim \Pi_r B_n&\text{(by \eqref{ZrcapBs})}\\&=(\dim B_{n-1}+1)-(\dim \Pi_r B_{n-1}+1)&\text{(by \eqref{eqdimPirB})}\\&=\dim B_{n-1}-\dim \Pi_r B_{n-1}&\\&=\dim Z_r\cap B_{n-1}&\text{(by \eqref{ZrcapBs})}\\&=|\mathcal{P}\cap ([1,r]\times [1,n-1])|&\text{(by induction)}\\&=|\mathcal{P}\cap ([1,r]\times [1,n])|&\text{(by the fact that $p_n\notin [1,r]\times [1,n])$}.\end{align*}

If $r\ge b_n$, we have
\begin{align*}\dim& Z_r\cap B_n&\\&=\dim B_n-\dim \Pi_r B_n&\text{(by \eqref{ZrcapBs})}\\&=(\dim B_{n-1}+1)-\dim \Pi_r B_{n-1}&\text{(by \eqref{eqdimPirB})}\\&=\dim Z_r\cap B_{n-1}+1&\text{(by \eqref{ZrcapBs})}\\&=|\mathcal{P}\cap ([1,r]\times [1,n-1])|+1&\text{(by induction)}\\&=|\mathcal{P}\cap ([1,r]\times [1,n])|&\text{(by the fact that $p_n\in [1,r]\times [1,n])$}.\end{align*}

Given $r\in\mathbb{Z}_+$, consider the set of vectors $\{v_p\,:\,p\in\mathcal{P}\cap ([1,r]\times [1,n])\}$. These vectors are linearly independent by part~\eqref{lemmdef0parta}, and they are contained in $Z_r\cap B_n$ by Lemma~\ref{banddgood}. Since the number of vectors matches the dimension of $Z_r\cap B_n$, it follows that they form a basis of $Z_r\cap B_n$. This proves part~\eqref{lemmdef0partb}.

Finally, if $b_n=-\infty$, the statement follows easily from the fact that in this case, we have $B_n=B_{n-1}$ and $Z_r\cap B_n=Z_r\cap B_{n-1}$ for all $r$.
\end{proof}

\subsection{Sampling the persistence diagram}

We define a procedure which we call \textsf{NextLine}. It takes a vector $v\in \{0,1\}^{n-1}$ as an input, and gives back a pair $(w,i)$ as an output, where $w\in \{0,1\}^{n}$ and $i\in\{1,2\dots,n\}\cup\{-\infty\}$. The output also depends on some additional randomness, which is provided by a coin which comes up heads with probability $\frac{q-1}q$. We assume that the coin tosses are independent. Given an input vector $v\in \{0,1\}^{n-1}$, let $u\in \{0,1\}^{n}$ be obtained from $v$ by appending it with a $0$. We go through the components of $u$ starting with $u_{n}$ and then moving from right to left. Whenever we encounter a component that is equal to $0$, we toss the coin and if it comes up heads then we stop, otherwise we continue moving to the left. Let us assume that we stopped at the $i$th component. If we went through all the components of $u$ without stopping, then we set $i=-\infty$. If $i\neq -\infty$, we output the pair $(w,i)$, where $w$ is obtained from $u$ by changing the $i$th component of $u$ to $1$ from $0$. If $i=-\infty$, we output $(u,-\infty)$. See Figure~\ref{Nexlinefigure}.

\begin{figure}[h!]
 
Example 1: 
\medskip

 \begin{tabular}{p{3.8cm} p{0.7cm} p{0.3cm} p{0.3cm} p{0.3cm} p{0.3cm} p{0.3cm} p{0.3cm} p{0.3cm} p{0.3cm} p{0.3cm} p{0.3cm} p{0.3cm} p{0.3cm} p{0.3cm} p{0.3cm} p{0.3cm} p{0.3cm} p{0.3cm}}
\cline{3-16}
The input vector $v$: &&\multicolumn{1}{|c|}{0}&\multicolumn{1}{|c|}{0}&\multicolumn{1}{|c}{1}&\multicolumn{1}{|c|}{1}&\multicolumn{1}{|c|}{0}&\multicolumn{1}{|c|}{0}&\multicolumn{1}{|c|}{0}&\multicolumn{1}{|c|}{1}&\multicolumn{1}{|c|}{1}&\multicolumn{1}{|c|}{0}&\multicolumn{1}{|c|}{1}&\multicolumn{1}{|c|}{0}&\multicolumn{1}{|c|}{1}&\multicolumn{1}{|c|}{0}&\multicolumn{1}{|c}{ }&&\\\cline{3-16}
&&\multicolumn{15}{c}{$\leftarrow\quad\leftarrow\quad\leftarrow\quad\leftarrow\quad$ Reading direction $\quad\leftarrow\quad\leftarrow\quad\leftarrow\quad\leftarrow$}\\
\cline{3-17}
The vector $u$: &&\multicolumn{1}{|c|}{0}&\multicolumn{1}{|c|}{0}&\multicolumn{1}{|c}{1}&\multicolumn{1}{|c|}{1}&\multicolumn{1}{|c|}{0}&\multicolumn{1}{|c|}{0}&\multicolumn{1}{|c|}{0}&\multicolumn{1}{|c|}{1}&\multicolumn{1}{|c|}{1}&\multicolumn{1}{|c|}{0}&\multicolumn{1}{|c|}{1}&\multicolumn{1}{|c|}{0}&\multicolumn{1}{|c|}{1}&\multicolumn{1}{|c|}{0}&\multicolumn{1}{|c|}{0}&\multicolumn{1}{|c}{ }&\\
\cline{3-17}
The coin tosses: &&&&&&&\multicolumn{1}{c}{H}&\multicolumn{1}{c}{T}&&&\multicolumn{1}{c}{T}&&\multicolumn{1}{c}{T}&&\multicolumn{1}{c}{T}&\multicolumn{1}{c}{T}&\\
\cline{3-17}
The output vector~$w$: &&\multicolumn{1}{|c|}{0}&\multicolumn{1}{|c|}{0}&\multicolumn{1}{|c}{1}&\multicolumn{1}{|c|}{1}&\multicolumn{1}{|c|}{0}&\multicolumn{1}{|c|}{\cellcolor{blue!20}1}&\multicolumn{1}{|c|}{0}&\multicolumn{1}{|c|}{1}&\multicolumn{1}{|c|}{1}&\multicolumn{1}{|c|}{0}&\multicolumn{1}{|c|}{1}&\multicolumn{1}{|c|}{0}&\multicolumn{1}{|c|}{1}&\multicolumn{1}{|c|}{0}&\multicolumn{1}{|c|}{0}&\multicolumn{1}{|c}{ }&\\\cline{3-17}
The output index $i$:&&&&&&&\multicolumn{1}{c}{6}\\
\end{tabular}

\bigskip

\bigskip

Example 2:

\medskip

\begin{tabular}{p{3.8cm} p{0.7cm} p{0.3cm} p{0.3cm} p{0.3cm} p{0.3cm} p{0.3cm} p{0.3cm} p{0.3cm} p{0.3cm} p{0.3cm} p{0.3cm} p{0.3cm} p{0.3cm} p{0.3cm} p{0.3cm} p{0.3cm} p{0.3cm} p{0.3cm}}
\cline{3-16}
The input vector $v$: &&\multicolumn{1}{|c|}{1}&\multicolumn{1}{|c|}{0}&\multicolumn{1}{|c}{1}&\multicolumn{1}{|c|}{0}&\multicolumn{1}{|c|}{0}&\multicolumn{1}{|c|}{0}&\multicolumn{1}{|c|}{1}&\multicolumn{1}{|c|}{1}&\multicolumn{1}{|c|}{0}&\multicolumn{1}{|c|}{1}&\multicolumn{1}{|c|}{0}&\multicolumn{1}{|c|}{1}&\multicolumn{1}{|c|}{0}&\multicolumn{1}{|c|}{1}&\multicolumn{1}{|c}{ }&&\\\cline{3-16}

&&\multicolumn{15}{c}{$\leftarrow\quad\leftarrow\quad\leftarrow\quad\leftarrow\quad$ Reading direction $\quad\leftarrow\quad\leftarrow\quad\leftarrow\quad\leftarrow$}\\\cline{3-17}
The vector $u$: &&\multicolumn{1}{|c|}{1}&\multicolumn{1}{|c|}{0}&\multicolumn{1}{|c|}{1}&\multicolumn{1}{|c|}{0}&\multicolumn{1}{|c|}{0}&\multicolumn{1}{|c|}{0}&\multicolumn{1}{|c|}{1}&\multicolumn{1}{|c|}{1}&\multicolumn{1}{|c|}{0}&\multicolumn{1}{|c|}{1}&\multicolumn{1}{|c|}{0}&\multicolumn{1}{|c|}{1}&\multicolumn{1}{|c|}{0}&\multicolumn{1}{|c|}{1}&\multicolumn{1}{|c|}{0}&\multicolumn{1}{|c}{ }&\\
\cline{3-17}
The coin tosses: &&&\multicolumn{1}{c}{T}&&\multicolumn{1}{c}{T}&\multicolumn{1}{c}{T}&\multicolumn{1}{c}{T}&&&\multicolumn{1}{c}{T}&&\multicolumn{1}{c}{T}&&\multicolumn{1}{c}{T}&&\multicolumn{1}{c}{T}&\\
\cline{3-17}
The output vector~$w$: &&\multicolumn{1}{|c|}{1}&\multicolumn{1}{|c|}{0}&\multicolumn{1}{|c|}{1}&\multicolumn{1}{|c|}{0}&\multicolumn{1}{|c|}{0}&\multicolumn{1}{|c|}{0}&\multicolumn{1}{|c|}{1}&\multicolumn{1}{|c|}{1}&\multicolumn{1}{|c|}{0}&\multicolumn{1}{|c|}{1}&\multicolumn{1}{|c|}{0}&\multicolumn{1}{|c|}{1}&\multicolumn{1}{|c|}{0}&\multicolumn{1}{|c|}{1}&\multicolumn{1}{|c|}{0}&\multicolumn{1}{|c}{ }&\\
\cline{3-17}
The output index $i$:&$-\infty$\\
\end{tabular}
\caption{Two sample input-output pairs for the procedure \textsf{NextLine}}
 \label{Nexlinefigure}
\end{figure}

The next lemma and its proof rely on the matrix $A$ defined in~\eqref{eqAdef}, the index $b_n$ defined in~\eqref{bidef2}, and their properties discussed in Section~\ref{SecFirstSteps}.

\begin{lemma}\label{lemmaNextLine}
 Let $v\in \{0,1\}^{n-1}$. Conditioned on the event that 
\[\left(A(1,n-1),A(2,n-1),\dots,A(n-1,n-1)\right)=v,\] the pair
 $\left(\left( A(1,n),A(2,n),\dots,A(n,n)\right),b_n\right)$ has the same distribution as the output \break of \textsf{NextLine}$(v)$. 
\end{lemma}
\begin{proof}
 Let $z_n$ be the $n$th row of $L$. We can write $z_n$ as $\sum_{j=1}^n \gamma_j e_j$, where $e_1,e_2,\dots$ is the standard basis of $\Finf$, and $\gamma_1,\dots,\gamma_n$ are i.i.d uniform random elements of $\mathbb{F}_q$. We will reveal the coefficients $\gamma_n,\gamma_{n-1},\dots,\gamma_1$ one-by-one in this order. Assume that we have already revealed $\gamma_n,\gamma_{n-1},\dots, \gamma_{i+1}$ and we have found that $\Pi_i z_n=\sum_{j=i+1}^n \gamma_j e_j\in \Pi_i B_{n-1}$. We distinguish two cases based on the value of $v_i=A(i,n-1)=\dim \Pi_{i-1} B_{n-1}-\dim \Pi_i B_{n-1}$ which is either $0$ or $1$. 
 \begin{itemize}
 \item If $v_i=1$, then $\beta e_i+\sum_{j=i+1}^n \gamma_j e_j\in  \Pi_{i-1} B_{n-1}$ for all choices of $\beta$. 
 \item If $v_i=0$, then there is a unique choice of $\beta$ such that $\beta e_i+\sum_{j=i+1}^n \gamma_j e_j\in  \Pi_{i-1} B_{n-1}$. 
 \end{itemize} 
 Thus,
\[\mathbb{P}\left(\sum_{j=i}^n \gamma_j e_j\in  \Pi_{i-1} B_{n-1}\,\Big|\,\sum_{j=i+1}^n \gamma_j e_j\in  \Pi_{i} B_{n-1}\right)=\begin{cases}
 1&\text{ if }v_i=1,\\
 \frac{1}q&\text{ if }v_i=0.
 \end{cases}\]
 The statement follows easily from this observation.
\end{proof}

Once we have Lemma~\ref{lemmaNextLine}, we can sample the matrix $A$ together with the sequence $b_1,b_2,\dots$ as follows. For all $n=1,2,\dots$, we set
\[\left(\left( A(1,n),A(2,n),\dots,A(n,n)\right),b_n\right)=\textsf{NextLine}(A(1,n-1),A(2,n-1),\dots,A(n-1,n-1)).\]

We refer to this algorithm as \textsf{SamplingA}. If we only sample $(A(i,j))_{1\le i\le j\le n}$ by running this algorithm until time $n$, we will use the notation \textsf{SamplingA}($n$).

Given $b_1,b_2,\dots$, we can obtain $\mathcal{P}$ as\[\mathcal{P}=\{(b_n,n)\,:\,n\in\mathbb{Z}_+, b_n\neq-\infty\}.\]

% The set $\mathcal{P}\cap [1,n]^2$ coincides with the set $\mathcal{P}_n$ given by Lemma~\ref{lemmaperind}. Thus, as a corollary of Lemma~\ref{lemmaperind}, we obtain the following lemma.
% \begin{lemma}\label{lemmadef0}
% There exist $\mathcal{P}\subset \overline{\Delta}$ and linearly independent vectors $(v_p)_{p\in \mathcal{P}}$ in $\Finf$ such that 
% \begin{itemize}
% \item for all $r,s\in \mathbb{Z}_+$, the vectors $\{v_p\,:\,p\in\mathcal{P}\cap([1,r]\times [1,s]) \}$ form a basis of $Z_r\cap B_s$;
% \item for all $p=(b,d)\in\mathcal{P}$, we have $b(v_p)=b$ and $d(v_p)=d$;
% \item for all $d$ there are at most one $b$ such that $(b,d)\in \mathcal{P}$.
% \end{itemize}
% Moreover, we can sample $\mathcal{P}$ using the \textsf{SamplingA} procedure. 
% \end{lemma}

Let 
\[N_n=|\{i\in \{1,2,\dots,n\}\,:\,A(i,n)=0\}|.\]
Looking at the algorithm \textsf{SamplingA}, one can see that $N_n$ is a Markov-chain, and
\begin{equation}\label{eqNntransition}\mathbb{P}(N_n=k\,|\,N_{n-1}=h)=\begin{cases} q^{-(h+1)}&\text{ if $k=h+1$},\\
1-q^{-(h+1)}&\text{ if $k=h$},\\
0&\text{ otherwise.}
\end{cases}\end{equation}
Moreover, the event $N_n=N_{n-1}+1$ coincides with the event that $b_n=-\infty$.

Let
\begin{align*}
T_0&=0,\\
T_i&=\min \{n\in \mathbb{Z}_+\,:\,N_n=i\}&\text{ for }i>0.
\end{align*}
Or equivalently, we can define $(T_i)$ by
\begin{align*}
T_0&=0,\\
T_i&=\min \{n>T_{i-1}\,:\,b_n=-\infty\}&\text{ for }i>0.
\end{align*}

The next lemma is a straightforward corollary of~\eqref{eqNntransition}.
\begin{lemma}\label{lemmaincrement}
 The increments $T_1-T_0,T_2-T_1,T_3-T_2,\dots$ are independent. Moreover, $T_i-T_{i-1}$ is a shifted geometric random variable with success probability $q^{-i}$.
\end{lemma} 

In particular, this lemma implies $T_i<\infty$ for all $i$ almost surely.

\begin{figure}[h]
\centering
\begin{tabular}{ p{0.3cm} p{0.3cm} p{0.3cm} p{0.3cm} p{0.3cm} p{0.3cm} p{0.3cm} p{0.3cm} p{0.3cm} p{0.3cm} p{0.3cm} p{0.3cm} p{0.3cm} p{0.3cm} p{0.3cm} p{0.3cm} p{0.3cm} p{0.3cm} p{0.3cm} p{0.3cm} p{0.3cm} p{0.3cm} p{0cm}}
\cline{1-22}\multicolumn{1}{|c|}{\cellcolor{yellow!50}1}&\multicolumn{1}{|c|}{\cellcolor{yellow!50}1}&\multicolumn{1}{|c|}{\cellcolor{yellow!50}1}&\multicolumn{1}{|c|}{0}&\multicolumn{1}{|c|}{\cellcolor{yellow!50}1}&\multicolumn{1}{|c|}{\cellcolor{yellow!50}1}&\multicolumn{1}{|c|}{\cellcolor{yellow!50}1}&\multicolumn{1}{|c|}{\cellcolor{yellow!50}1}&\multicolumn{1}{|c|}{\cellcolor{yellow!50}1}&\multicolumn{1}{|c|}{\cellcolor{yellow!50}1}&\multicolumn{1}{|c|}{\cellcolor{yellow!50}1}&\multicolumn{1}{|c|}{\cellcolor{yellow!50}1}&\multicolumn{1}{|c|}{\cellcolor{yellow!50}1}&\multicolumn{1}{|c|}{\cellcolor{yellow!50}1}&\multicolumn{1}{|c|}{\cellcolor{yellow!50}1}&\multicolumn{1}{|c|}{\cellcolor{yellow!50}1}&\multicolumn{1}{|c|}{0}&\multicolumn{1}{|c|}{\cellcolor{yellow!50}1}&\multicolumn{1}{|c|}{\cellcolor{yellow!50}1}&\multicolumn{1}{|c|}{\cellcolor{yellow!50}1}&\multicolumn{1}{|c|}{0}&\multicolumn{1}{|c|}{\cellcolor{blue!20}1}&\\\cline{1-22}
\multicolumn{1}{|c|}{\cellcolor{yellow!50}1}&\multicolumn{1}{|c|}{\cellcolor{yellow!50}1}&\multicolumn{1}{|c|}{\cellcolor{yellow!50}1}&\multicolumn{1}{|c|}{0}&\multicolumn{1}{|c|}{\cellcolor{yellow!50}1}&\multicolumn{1}{|c|}{\cellcolor{yellow!50}1}&\multicolumn{1}{|c|}{\cellcolor{yellow!50}1}&\multicolumn{1}{|c|}{\cellcolor{yellow!50}1}&\multicolumn{1}{|c|}{\cellcolor{yellow!50}1}&\multicolumn{1}{|c|}{\cellcolor{yellow!50}1}&\multicolumn{1}{|c|}{\cellcolor{yellow!50}1}&\multicolumn{1}{|c|}{\cellcolor{yellow!50}1}&\multicolumn{1}{|c|}{\cellcolor{yellow!50}1}&\multicolumn{1}{|c|}{\cellcolor{yellow!50}1}&\multicolumn{1}{|c|}{\cellcolor{yellow!50}1}&\multicolumn{1}{|c|}{\cellcolor{yellow!50}1}&\multicolumn{1}{|c|}{0}&\multicolumn{1}{|c|}{\cellcolor{yellow!50}1}&\multicolumn{1}{|c|}{\cellcolor{blue!20}1}&\multicolumn{1}{|c|}{\cellcolor{yellow!50}1}&\multicolumn{1}{|c|}{0}&\\\cline{1-21}
\multicolumn{1}{|c|}{\cellcolor{yellow!50}1}&\multicolumn{1}{|c|}{\cellcolor{yellow!50}1}&\multicolumn{1}{|c|}{\cellcolor{yellow!50}1}&\multicolumn{1}{|c|}{0}&\multicolumn{1}{|c|}{\cellcolor{yellow!50}1}&\multicolumn{1}{|c|}{\cellcolor{yellow!50}1}&\multicolumn{1}{|c|}{\cellcolor{yellow!50}1}&\multicolumn{1}{|c|}{\cellcolor{yellow!50}1}&\multicolumn{1}{|c|}{\cellcolor{yellow!50}1}&\multicolumn{1}{|c|}{\cellcolor{yellow!50}1}&\multicolumn{1}{|c|}{\cellcolor{yellow!50}1}&\multicolumn{1}{|c|}{\cellcolor{yellow!50}1}&\multicolumn{1}{|c|}{\cellcolor{yellow!50}1}&\multicolumn{1}{|c|}{\cellcolor{yellow!50}1}&\multicolumn{1}{|c|}{\cellcolor{yellow!50}1}&\multicolumn{1}{|c|}{\cellcolor{yellow!50}1}&\multicolumn{1}{|c|}{0}&\multicolumn{1}{|c|}{\cellcolor{yellow!50}1}&\multicolumn{1}{|c|}{0}&\multicolumn{1}{|c|}{\cellcolor{blue!20}1}&\\\cline{1-20}
\multicolumn{1}{|c|}{\cellcolor{yellow!50}1}&\multicolumn{1}{|c|}{\cellcolor{yellow!50}1}&\multicolumn{1}{|c|}{\cellcolor{yellow!50}1}&\multicolumn{1}{|c|}{0}&\multicolumn{1}{|c|}{\cellcolor{yellow!50}1}&\multicolumn{1}{|c|}{\cellcolor{yellow!50}1}&\multicolumn{1}{|c|}{\cellcolor{yellow!50}1}&\multicolumn{1}{|c|}{\cellcolor{yellow!50}1}&\multicolumn{1}{|c|}{\cellcolor{yellow!50}1}&\multicolumn{1}{|c|}{\cellcolor{yellow!50}1}&\multicolumn{1}{|c|}{\cellcolor{yellow!50}1}&\multicolumn{1}{|c|}{\cellcolor{yellow!50}1}&\multicolumn{1}{|c|}{\cellcolor{yellow!50}1}&\multicolumn{1}{|c|}{\cellcolor{yellow!50}1}&\multicolumn{1}{|c|}{\cellcolor{blue!20}1}&\multicolumn{1}{|c|}{\cellcolor{yellow!50}1}&\multicolumn{1}{|c|}{0}&\multicolumn{1}{|c|}{\cellcolor{yellow!50}1}&\multicolumn{1}{|c|}{0}&\\\cline{1-19}
\multicolumn{1}{|c|}{\cellcolor{yellow!50}1}&\multicolumn{1}{|c|}{\cellcolor{yellow!50}1}&\multicolumn{1}{|c|}{\cellcolor{yellow!50}1}&\multicolumn{1}{|c|}{0}&\multicolumn{1}{|c|}{\cellcolor{yellow!50}1}&\multicolumn{1}{|c|}{\cellcolor{yellow!50}1}&\multicolumn{1}{|c|}{\cellcolor{yellow!50}1}&\multicolumn{1}{|c|}{\cellcolor{yellow!50}1}&\multicolumn{1}{|c|}{\cellcolor{yellow!50}1}&\multicolumn{1}{|c|}{\cellcolor{yellow!50}1}&\multicolumn{1}{|c|}{\cellcolor{yellow!50}1}&\multicolumn{1}{|c|}{\cellcolor{yellow!50}1}&\multicolumn{1}{|c|}{\cellcolor{yellow!50}1}&\multicolumn{1}{|c|}{\cellcolor{yellow!50}1}&\multicolumn{1}{|c|}{0}&\multicolumn{1}{|c|}{\cellcolor{yellow!50}1}&\multicolumn{1}{|c|}{0}&\multicolumn{1}{|c|}{\cellcolor{blue!20}1}&\\\cline{1-18}
\multicolumn{1}{|c|}{\cellcolor{yellow!50}1}&\multicolumn{1}{|c|}{\cellcolor{yellow!50}1}&\multicolumn{1}{|c|}{\cellcolor{yellow!50}1}&\multicolumn{1}{|c|}{0}&\multicolumn{1}{|c|}{\cellcolor{yellow!50}1}&\multicolumn{1}{|c|}{\cellcolor{yellow!50}1}&\multicolumn{1}{|c|}{\cellcolor{yellow!50}1}&\multicolumn{1}{|c|}{\cellcolor{yellow!50}1}&\multicolumn{1}{|c|}{\cellcolor{yellow!50}1}&\multicolumn{1}{|c|}{\cellcolor{yellow!50}1}&\multicolumn{1}{|c|}{\cellcolor{yellow!50}1}&\multicolumn{1}{|c|}{\cellcolor{yellow!50}1}&\multicolumn{1}{|c|}{\cellcolor{yellow!50}1}&\multicolumn{1}{|c|}{\cellcolor{yellow!50}1}&\multicolumn{1}{|c|}{0}&\multicolumn{1}{|c|}{\cellcolor{blue!20}1}&\multicolumn{1}{|c|}{0}&\\\cline{1-17}
\multicolumn{1}{|c|}{\cellcolor{yellow!50}1}&\multicolumn{1}{|c|}{\cellcolor{yellow!50}1}&\multicolumn{1}{|c|}{\cellcolor{yellow!50}1}&\multicolumn{1}{|c|}{0}&\multicolumn{1}{|c|}{\cellcolor{blue!20}1}&\multicolumn{1}{|c|}{\cellcolor{yellow!50}1}&\multicolumn{1}{|c|}{\cellcolor{yellow!50}1}&\multicolumn{1}{|c|}{\cellcolor{yellow!50}1}&\multicolumn{1}{|c|}{\cellcolor{yellow!50}1}&\multicolumn{1}{|c|}{\cellcolor{yellow!50}1}&\multicolumn{1}{|c|}{\cellcolor{yellow!50}1}&\multicolumn{1}{|c|}{\cellcolor{yellow!50}1}&\multicolumn{1}{|c|}{\cellcolor{yellow!50}1}&\multicolumn{1}{|c|}{\cellcolor{yellow!50}1}&\multicolumn{1}{|c|}{0}&\multicolumn{1}{|c|}{0}&\\\cline{1-16}
\multicolumn{1}{|c|}{\cellcolor{yellow!50}1}&\multicolumn{1}{|c|}{\cellcolor{yellow!50}1}&\multicolumn{1}{|c|}{\cellcolor{yellow!50}1}&\multicolumn{1}{|c|}{0}&\multicolumn{1}{|c|}{0}&\multicolumn{1}{|c|}{\cellcolor{yellow!50}1}&\multicolumn{1}{|c|}{\cellcolor{yellow!50}1}&\multicolumn{1}{|c|}{\cellcolor{yellow!50}1}&\multicolumn{1}{|c|}{\cellcolor{yellow!50}1}&\multicolumn{1}{|c|}{\cellcolor{yellow!50}1}&\multicolumn{1}{|c|}{\cellcolor{yellow!50}1}&\multicolumn{1}{|c|}{\cellcolor{yellow!50}1}&\multicolumn{1}{|c|}{\cellcolor{yellow!50}1}&\multicolumn{1}{|c|}{\cellcolor{blue!20}1}&\multicolumn{1}{|c|}{0}&\\\cline{1-15}
\multicolumn{1}{|c|}{\cellcolor{yellow!50}1}&\multicolumn{1}{|c|}{\cellcolor{yellow!50}1}&\multicolumn{1}{|c|}{\cellcolor{yellow!50}1}&\multicolumn{1}{|c|}{0}&\multicolumn{1}{|c|}{0}&\multicolumn{1}{|c|}{\cellcolor{yellow!50}1}&\multicolumn{1}{|c|}{\cellcolor{yellow!50}1}&\multicolumn{1}{|c|}{\cellcolor{yellow!50}1}&\multicolumn{1}{|c|}{\cellcolor{yellow!50}1}&\multicolumn{1}{|c|}{\cellcolor{yellow!50}1}&\multicolumn{1}{|c|}{\cellcolor{yellow!50}1}&\multicolumn{1}{|c|}{\cellcolor{yellow!50}1}&\multicolumn{1}{|c|}{\cellcolor{blue!20}1}&\multicolumn{1}{|c|}{0}&\\\cline{1-14}
\multicolumn{1}{|c|}{\cellcolor{yellow!50}1}&\multicolumn{1}{|c|}{\cellcolor{yellow!50}1}&\multicolumn{1}{|c|}{\cellcolor{yellow!50}1}&\multicolumn{1}{|c|}{0}&\multicolumn{1}{|c|}{0}&\multicolumn{1}{|c|}{\cellcolor{yellow!50}1}&\multicolumn{1}{|c|}{\cellcolor{yellow!50}1}&\multicolumn{1}{|c|}{\cellcolor{blue!20}1}&\multicolumn{1}{|c|}{\cellcolor{yellow!50}1}&\multicolumn{1}{|c|}{\cellcolor{yellow!50}1}&\multicolumn{1}{|c|}{\cellcolor{yellow!50}1}&\multicolumn{1}{|c|}{\cellcolor{yellow!50}1}&\multicolumn{1}{|c|}{0}&\\\cline{1-13}
\multicolumn{1}{|c|}{\cellcolor{yellow!50}1}&\multicolumn{1}{|c|}{\cellcolor{yellow!50}1}&\multicolumn{1}{|c|}{\cellcolor{yellow!50}1}&\multicolumn{1}{|c|}{0}&\multicolumn{1}{|c|}{0}&\multicolumn{1}{|c|}{\cellcolor{yellow!50}1}&\multicolumn{1}{|c|}{\cellcolor{yellow!50}1}&\multicolumn{1}{|c|}{0}&\multicolumn{1}{|c|}{\cellcolor{yellow!50}1}&\multicolumn{1}{|c|}{\cellcolor{yellow!50}1}&\multicolumn{1}{|c|}{\cellcolor{yellow!50}1}&\multicolumn{1}{|c|}{\cellcolor{blue!20}1}&\\\cline{1-12}
\multicolumn{1}{|c|}{\cellcolor{yellow!50}1}&\multicolumn{1}{|c|}{\cellcolor{yellow!50}1}&\multicolumn{1}{|c|}{\cellcolor{yellow!50}1}&\multicolumn{1}{|c|}{0}&\multicolumn{1}{|c|}{0}&\multicolumn{1}{|c|}{\cellcolor{yellow!50}1}&\multicolumn{1}{|c|}{\cellcolor{yellow!50}1}&\multicolumn{1}{|c|}{0}&\multicolumn{1}{|c|}{\cellcolor{yellow!50}1}&\multicolumn{1}{|c|}{\cellcolor{yellow!50}1}&\multicolumn{1}{|c|}{\cellcolor{blue!20}1}&\\\cline{1-11}
\multicolumn{1}{|c|}{\cellcolor{yellow!50}1}&\multicolumn{1}{|c|}{\cellcolor{yellow!50}1}&\multicolumn{1}{|c|}{\cellcolor{yellow!50}1}&\multicolumn{1}{|c|}{0}&\multicolumn{1}{|c|}{0}&\multicolumn{1}{|c|}{\cellcolor{yellow!50}1}&\multicolumn{1}{|c|}{\cellcolor{yellow!50}1}&\multicolumn{1}{|c|}{0}&\multicolumn{1}{|c|}{\cellcolor{yellow!50}1}&\multicolumn{1}{|c|}{\cellcolor{blue!20}1}&\\\cline{1-10}
\multicolumn{1}{|c|}{\cellcolor{yellow!50}1}&\multicolumn{1}{|c|}{\cellcolor{yellow!50}1}&\multicolumn{1}{|c|}{\cellcolor{yellow!50}1}&\multicolumn{1}{|c|}{0}&\multicolumn{1}{|c|}{0}&\multicolumn{1}{|c|}{\cellcolor{yellow!50}1}&\multicolumn{1}{|c|}{\cellcolor{yellow!50}1}&\multicolumn{1}{|c|}{0}&\multicolumn{1}{|c|}{\cellcolor{blue!20}1}&\\\cline{1-9}
\multicolumn{1}{|c|}{\cellcolor{yellow!50}1}&\multicolumn{1}{|c|}{\cellcolor{yellow!50}1}&\multicolumn{1}{|c|}{\cellcolor{yellow!50}1}&\multicolumn{1}{|c|}{0}&\multicolumn{1}{|c|}{0}&\multicolumn{1}{|c|}{\cellcolor{yellow!50}1}&\multicolumn{1}{|c|}{\cellcolor{yellow!50}1}&\multicolumn{1}{|c|}{0}&\\\cline{1-8}
\multicolumn{1}{|c|}{\cellcolor{yellow!50}1}&\multicolumn{1}{|c|}{\cellcolor{yellow!50}1}&\multicolumn{1}{|c|}{\cellcolor{yellow!50}1}&\multicolumn{1}{|c|}{0}&\multicolumn{1}{|c|}{0}&\multicolumn{1}{|c|}{\cellcolor{yellow!50}1}&\multicolumn{1}{|c|}{\cellcolor{blue!20}1}&\\\cline{1-7}
\multicolumn{1}{|c|}{\cellcolor{yellow!50}1}&\multicolumn{1}{|c|}{\cellcolor{yellow!50}1}&\multicolumn{1}{|c|}{\cellcolor{yellow!50}1}&\multicolumn{1}{|c|}{0}&\multicolumn{1}{|c|}{0}&\multicolumn{1}{|c|}{\cellcolor{blue!20}1}&\\\cline{1-6}
\multicolumn{1}{|c|}{\cellcolor{yellow!50}1}&\multicolumn{1}{|c|}{\cellcolor{blue!20}1}&\multicolumn{1}{|c|}{\cellcolor{yellow!50}1}&\multicolumn{1}{|c|}{0}&\multicolumn{1}{|c|}{0}&\\\cline{1-5}
\multicolumn{1}{|c|}{\cellcolor{yellow!50}1}&\multicolumn{1}{|c|}{0}&\multicolumn{1}{|c|}{\cellcolor{blue!20}1}&\multicolumn{1}{|c|}{0}&\\\cline{1-4}
\multicolumn{1}{|c|}{\cellcolor{yellow!50}1}&\multicolumn{1}{|c|}{0}&\multicolumn{1}{|c|}{0}&\\\cline{1-3}
\multicolumn{1}{|c|}{\cellcolor{blue!20}1}&\multicolumn{1}{|c|}{0}&\\\cline{1-2}
\multicolumn{1}{|c|}{0}&\\\cline{1-1}
\end{tabular}

\caption{A random sample of the matrix $A$ with the choice of $q=2$. The blue cells correspond to elements of the verbose persistence diagram $\mathcal{P}$. The orientation of the picture is chosen such that it is in alignment with the persistence diagram, that is, the bottom left cell displays $A(1,1)$, the cell above that displays~$A(1,2)$.}
\end{figure}
\subsection{A multi phase sampling algorithm}

The following lemma is clear from the description of the \textsf{SamplingA} algorithm.

\begin{lemma}\label{invarianceindist}
Given $r\in \mathbb{Z}_{\ge 0}$, let us define the random matrix $A'=\left(A'(i,j)\right)_{i,j\in \mathbb{Z_+}}$ by setting
\[A'(i,j)=A(i+r,j+r).\]
Then $A'$ has the same distribution as $A$.
\end{lemma}

Next we define a variant of the \textsf{NextLine} procedure, which we call \textsf{NextLine2}. It takes a vector $v\in \{0,1\}^{n}$ as an input, and gives back a pair $(w,i)$ as an output, where $w\in \{0,1\}^{n}$ and $i\in\{1,2\dots,n\}\cup\{-\infty\}$. As before, the output also depends on some additional randomness, which is provided by a coin which comes up heads with probability $\frac{q-1}q$. We assume that the coin tosses are independent. We go through the components of $v$ starting with $v_{n}$ and then moving from right to left. Whenever we encounter a component that is equal to $0$, we toss the coin and if it comes up heads then we stop, otherwise we continue moving to the left. Let us assume that we stopped at the $i$th component. If we went through all the components of $v$ without stopping, then we set $i=-\infty$. If $i\neq -\infty$, we output the pair $(w,i)$, where $w$ is obtained from $v$ by changing the $i$th component of $v$ to~$1$ from $0$. If $i=-\infty$, we output $(v,-\infty)$. In other words, the only difference between \textsf{NextLine} and \textsf{NextLine2} is that while in \textsf{NextLine}, we start by appending the input with a zero, in \textsf{NextLine2}, this step is omitted.

Next, we describe a procedure called \textsf{MultiPhaseSamplingA}. This procedure takes an $r\in\mathbb{Z}_{\ge 0}$ as an input, and produces a random matrix $A=\left(A(i,j)\right)_{i,j\in\mathbb{Z}_+}$ in three phases as follows:

\begin{itemize}
 \item \textbf{Phase 1:}

 Using \textsf{SamplingA}, we generate a random matrix $A'=(A'(i,j))_{i,j\in \mathbb{Z}_+}$ together with the corresponding sequences of birth times $b_1',b_2',\dots$.

 Then, for all $i,j\in \mathbb{Z}_+$, we set
\[A(i+r,j+r)=A'(i,j).\]

 \item \textbf{Phase 2:} 

 We use \textsf{SamplingA}($r$) to generate $\left(A(i,j)\right)_{1\le i\le j\le r}$ and $b_1,b_2,\dots,b
 _r$.

 \item \textbf{Phase 3:}

 For $j=1,2,3,\dots$, we do the following:
 \begin{itemize}
 \item If $b_j'\neq -\infty$, we set $b_{j+r}=r+b_j'$, and
 \begin{multline*}(A(1,j+r),A(2,j+r),\dots,A(r,j+r))\\=(A(1,j+r-1),A(2,j+r-1),\dots,A(r,j+r-1)).\end{multline*}
 \item If $b_j'=-\infty$, then we set
 \begin{multline*}((A(1,j+r),A(2,j+r),\dots,A(r,j+r)),b_{j+r})\\=\textsf{NextLine2}(A(1,j+r-1),A(2,j+r-1),\dots,A(r,j+r-1))).\end{multline*}
 \end{itemize}
\end{itemize}

See Figure~\ref{FigureRegions} for an illustration.

\begin{figure}
 \centering
 \begin{tikzpicture}
 \draw[thick,black] (0,0)--(0,10);
 \draw[thick,black] (0,0)--(10,10);
 \draw[thick,black] (0,5)--(5,5);
 \draw[thick,black] (5,5)--(5,10);
 \node at (2,4) {Phase 2};
 \node at (2,9) {Phase 3};
 \node at (7,9) {Phase 1};
 \draw [thick, decorate,
 decoration = {brace,amplitude=10pt}] (-0.2,0) -- (-0.2,5) node[pos=0.5,left=14pt,black]{\Large $r$};;
 \end{tikzpicture}
 \caption{The regions of the matrix $A$ filled by the different phases \break of \textsf{MultiPhaseSamplingA}.}
 \label{FigureRegions}
\end{figure}
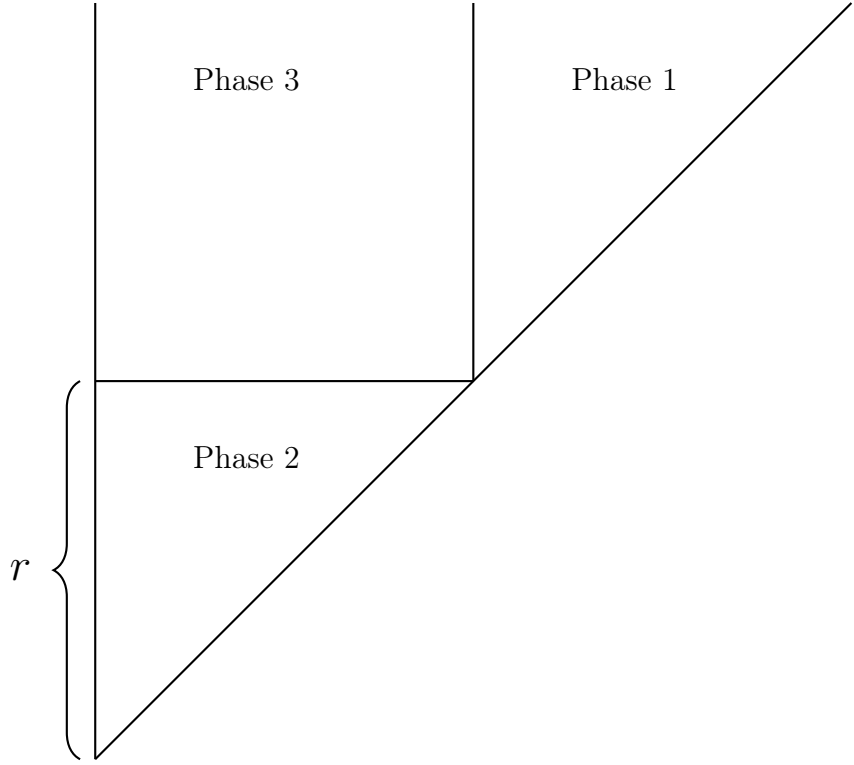

It is straightforward to prove the following lemma.

\begin{lemma}
 The random matrices produced by \textsf{SamplingA} and \textsf{MultiPhaseSamplingA} have the same distribution.
\end{lemma}

For $r,j\in \mathbb{Z}_{\ge 0}$, let
\[M_{r,j}=\left|\{i\,:\,1\le i\le r,\,A(i,j+r)=0\}\right|.\]

For a fixed $r$, let us try to understand the evolution of the sequence $M_{r,0},M_{r,1},\dots$ conditioned on the choices that were made in the first two phases of \textsf{MultiPhaseSamplingA}. Let $j\ge 1$, one can see that:
\begin{itemize}
 \item If $b_j'\neq -\infty$, then $M_{r,j}=M_{r,j-1}$.
 \item If $b_j'=-\infty$, then
\[\mathbb{P}(M_{r,j}=k\,|\,M_{r,0},M_{r,1},\dots,M_{r,j-1})=\begin{cases} q^{-M_{r,j-1}}&\text{ if $k=M_{r,j-1}$,}\\
 1-q^{-M_{r,j-1}}&\text{ if $k=M_{r,j-1}-1$,}\\
 0&\text{ otherwise.}
 \end{cases}
 \]
\end{itemize}

In agreement with our earlier definitions, let
\begin{align*}
T_0'&=0,\\
T_i'&=\min \{n>T_{i-1}'\,:\,b_n'=-\infty\}&\text{ for }i>0.
\end{align*}

Then our previous observations on the evolution of $M_{r,0},M_{r,1},\dots$ can be rephrases as follows:
\begin{lemma}\hfill\label{Markovlemma}
\begin{enumerate}[(a)]
 \item $M_{r,j}=M_{r,T_i'}$ for all $T_i'\le j<T_{i+1}'$.
 \item The sequence $M_{r,T_0'},M_{r,T_1'},M_{r,T_2'},\dots$ is a Markov-chain where the transition probabilities are given by 
\[P(i,j)=\begin{cases}
 q^{-i}&\text{if $j=i$},\\
 1-q^{-i}&\text{if $j=i-1$},\\
 0&\text{otherwise}.
 \end{cases}\]
\end{enumerate}
\end{lemma}
\begin{lemma}\label{lemmafinitedeath}
Given any $i\in \mathbb{Z}_+$, let
\[d_i=\min\{j\in\mathbb{Z}_+\,:\,(i,j)\in\mathcal{P}\}.\]
Then $d_i-i$ has the same distribution as $X$ defined in the statement of Theorem~\ref{thmlifetime}. 

In particular, with probability $1$, for all $i$, we have $d_i<\infty$, that is, there is a $d_i$ such that~$(i,d_i)\in \mathcal{P}$.

Moreover, $d_i$ is the unique $j$ such that $(i,j)\in\mathcal{P}$.

\end{lemma}
\begin{proof}
First, we prove the case $i=1$. Let us analyze \textsf{MultiPhaseSamplingA}($1$). We have
\[\mathbb{P}(M_{1,0}=0)=\mathbb{P}(A(1,1)=1)=\frac{q-1}q.\]
Moreover, by Lemma~\ref{Markovlemma}, we have
\[\mathbb{P}(M_{1,T_{i+1}'}=0\,|\,M_{1,T_i'}=1)=\frac{q-1}q.\]
Therefore, if we define
\[G=1+\min \{i\in \mathbb{Z}_{\ge 0}\,:\, M_{1,T_i'}=0\}.\]
Then $G$ is a shifted geometric random variable with success probability $\frac{q-1}q$. Recall that by Lemma~\ref{lemmaincrement}, $T_1'-T_0',T_2'-T_1',T_3'-T_2',\dots$ are independent and $T_i'-T_{i-1}'$ is a shifted geometric random variable with success probability $q^{-i}$. Moreover, observe that $G$ is also independent from~$T_1'-T_0',T_2'-T_1',T_3'-T_2',\dots$.

Thus,
\begin{multline*}d_1-1=\min\{j\in\mathbb{Z}_+\,:\,(1,j)\in\mathcal{P}\}-1=\min\{j\in \mathbb{Z}_+\,|\,A(1,j)=1\}-1\\=T_{G-1}'=\sum_{i=1}^{G-1} (T_i'-T_{i-1}')\end{multline*}
has the same distribution as $X$.

Using Lemma~\ref{invarianceindist}, we see that $d_i-i$ has the same distribution as $d_1-1$ for all $i$. Thus, $d_i-i$ has the same distribution as $X$.

It is straightforward to see from the \textsf{SamplingA} algorithm that $A(i,j)=1$ for all $j\ge d_i$, so for $j>d_i$, we can not have $(i,j)\in\mathcal{P}$. Thus, the last statement also follows.
\end{proof}

The next lemma is a straightforward consequence of our eralier lemmas.

\begin{lemma}\label{lemmaBasicProp}\hfill
\begin{enumerate}[(a)]
 \item Let $i\in\mathbb{Z}_+$. Then
\[A(i,k)=\begin{cases} 0&\text{if }k<d_i\\
 1&\text{if }k\ge d_i.\end{cases}\]
 \item The set $\mathcal{P}$ uniquely determines the matrix $A$. Moreover, the $\mathcal{P}\cap [1,n]^2$ uniquely determines the submatrix $(A(i,j))_{1\le i,j\le n}$.
 \item The matrix $A$ uniquely determines $\mathcal{P}$.
 \item\label{partdofBasic} Given $i\in \mathbb{Z}_{\ge 0}$ and $j\in\mathbb{Z}_+$, we have
\[|\{i<k\le j\,:\,A(k,j)=1\}|=|\{i<k\le j\,:\,d_k\le j\}|=|\{1\le h\le j\,:\,b_h>i\}|.\]
\end{enumerate}
\end{lemma}

The next lemma will paly a key role in understanding the fluctuations of the persistent Betti numbers.

\begin{lemma}\label{perbettiasrandommatrices}
Let $r,t\in\mathbb{Z}_{\ge 0}$. Let $K$ and $K'$ be independent random variables such that $K$ has the same distribution as $\dim\ker L_r$, and $K'$ has the same distribution as $\dim\ker L_t$. Let $H$ be a uniform random $K'\times K$ matrix over $\mathbb{F}_q$. 

Then $\beta^{r,r+t}$ has the same distribution as $\dim\ker H$.
\end{lemma}
\begin{proof}
Let us sample the random matrix $A$ by the procedure \textsf{MultiPhaseSamplingA}($r$). 

For a moment, let us redefine $K$ and $K'$ as follows:
\begin{align*}
 K&=|\{i\in \{1,2,\dots,r\}\,:\,A(i,r)=0\}|,\\
 K'&=|\{j\in \{1,2,\dots,t\}\,:\,b_j'=-\infty\}|.
\end{align*}
Since the choices made by \textsf{MultiPhaseSamplingA}($r$) are independent, we see that $K$ and $K'$ are independent. Moreover, by equation~\eqref{eqcorankexpr} and Lemma~\ref{lemmaBasicProp}, we see that $K$ has the same distribution as $\dim\ker L_r$, and $K'$ has the same distribution as $\dim\ker L_t$.

By Lemma~\ref{lemmaBasicProp}, we see that
\[\beta^{r,r+t}=|\{i\in\{1,2,\dots,r\}\,:\,d_i>r+t\}|=|\{i\in\{1,2,\dots,r\}\,:\,A(i,r+t)=0\}|.\]
Combining this with Lemma~\ref{Markovlemma}, we see that 
\[\mathbb{P}(\beta^{r,r+t}=j\,|\,K,K')=P^{K'}(K,j).\]

Let $H_i$ be the submatrix of $H$ consisting the first $i$ rows of $H$. Using simple linear algebra, one can see that $K=\dim\ker H_0, \dim\ker H_1,\dots,\dim\ker H_{K'}$ is also a Markov-chain with transition matrix $P$. Thus, the statement follows.
\end{proof}

\subsection{Finishing the proof of Lemma~\ref{lemmadef}}

Finally, we are ready to provide a full proof of Lemma~\ref{lemmadef}. Consider $\mathcal{P}$ and $(v_p)_{p\in\mathcal{P}}$ defined in Section~\ref{SecFirstSteps}. Part \eqref{lemmadedpart1} of Lemma~\ref{lemmadef} follows from Lemma~\ref{banddgood}. Part \eqref{lemmadedpart4b} follows directly from the construction of $\mathcal{P}$. Part \eqref{lemmadedpart2a} and part \eqref{lemmadedpart4a} were proved in Lemma~\ref{lemmadef0ind} and Lemma~\ref{lemmafinitedeath}, respectively. Part~\eqref{lemmadedpart3a} is a direct consequence of part \eqref{lemmadedpart2a}. Part~\eqref{lemmadedpart3b} follows from part~\eqref{lemmadedpart4a}. Combining this with the fact that the vectors $(v_p)_{p\in\mathcal{P}}$ are linearly independent, we see that part~\eqref{lemmadedpart2b} is also true. In particular, the span of $(v_p)_{p\in\mathcal{P}}$ contains $Z_n$ for all $n$. Thus, $(v_p)_{p\in\mathcal{P}}$ spans~$\Finf$. Therefore, the vectors $(v_p)_{p\in\mathcal{P}}$ are not only linearly independent, but indeed form a basis of~$\Finf$. Part~\eqref{lemmadedpart3c} follows by combining part \eqref{lemmadedpart3a} and part \eqref{lemmadedpart3b}. Finally, to prove part~\eqref{lemmadedpart5}, observe that if $\mathcal{P}$ satisfies~\eqref{lemmadedpart3a}, then $(r,s)\in \mathcal{P}$ if and only if
\[\dim Z_r\cap B_s-\dim Z_{r-1}\cap B_s-\dim Z_r\cap B_{s-1}+\dim Z_{r-1}\cap B_{s-1}=1.\]

\section{The proof of Theorem~\ref{thmexplicit}}

The first part of the statement follows directly from Lemma~\ref{lemmadef}, so we move on to prove the second statement. Given $S$, let $A_S=(A_S(i,j))_{1\le i,j\le n}$ be the matrix determined by $S$ as in Lemma~\ref{lemmaBasicProp}. Note that by Lemma~\ref{lemmaBasicProp}, the event $\mathcal{P}\cap [1,n]^2=S$ is the same as the event that~$A(i,j)=A_S(i,j)$ for all $1\le i,j\le n$. Let $\mathcal{A}_j$ be the event that $A_S(i,j)=A(i,j)$ for all $1\le i\le n$. By definition $\mathcal{A}_0$ has probability one.

Analyzing the \textsf{NextLine} procedure we see that, if $b_j(S)\neq -\infty$, then
\[\mathbb{P}(\mathcal{A}_{j}\,|\,\mathcal{A}_{j-1})=\frac{q-1}{q} \left(\frac{1}q\right)^{j-b_j(S)-\left|\{b_j(S)<i<j\,:\,A_S(i,j-1)=1\}\right|}.\]

Moreover, if $b_j=-\infty$, then
\[\mathbb{P}(\mathcal{A}_{j}\,|\,\mathcal{A}_{j-1})= \left(\frac{1}q\right)^{j-\left|\{0<i<j\,:\,A_S(i,j-1)=1\}\right|}.\]

Summarizing the above two cases, we see that
\[\mathbb{P}(\mathcal{A}_{j}\,|\,\mathcal{A}_{j-1})=\left(\frac{q-1}{q}\right)^{\mathbbm{1}(b_j(S)\neq-\infty)} \left(\frac{1}q\right)^{j-\max(b_j(S),0)-\left|\{\max(b_j(S),0)<i<j\,:\,A_S(i,j-1)=1\}\right|}.\]

Multiplying these equations, we see that
\begin{equation}\label{expf0}\mathbb{P}(\mathcal{P}\cap [1,n]^2=S)=\left(\frac{q-1}{q}\right)^{|S|}q^{\Sigma(S)-{{n+1}\choose{2}}+\sum_{j=2}^n \left|\{\max(b_j(S),0)<i<j\,:\,A_S(i,j-1)=1\}\right|}.\end{equation}

By part~\eqref{partdofBasic} of Lemma~\ref{lemmaBasicProp}, we see that
\begin{align*}\sum_{j=2}^n & \left|\{\max(b_j(S),0)<i<j\,:\,A_S(i,j-1)=1\}\right|\\&=\sum_{j=2}^n |\{1\le h\le j-1\,:\,b_h(S)>b_j(S)\}|\\&=\inv(S).\end{align*}

Combining this with \eqref{expf0}, Theorem~\ref{thmexplicit} follows.

\section{Law of large numbers for the distributions of lifetimes -- The proof of Theorem~\ref{thmlifetime}}

Note that
\[\frac{|\{(b,d)\in \mathcal{P}\cap [1,n]^2\,:\,d-b=k\}|}{n}=\frac{|\{i\in \{1,2,\dots,n\}\,:\,d_i-i=k\}|}n+o(1).\]

Combining this with Lemma~\ref{lemmafinitedeath}, we see that
\begin{equation}\label{eqExpLifetime}
 \lim_{n\to\infty} \mathbb{E}\frac{|\{(b,d)\in \mathcal{P}\cap [1,n]^2\,:\,d-b=k\}|}{n}=\lim_{n\to\infty} \frac{1}n\sum_{i=1}^n \mathbb{P}(d_i-i=k)=\mathbb{P}(X=k).
\end{equation}

For $j> k$, let us define
\[f^{(j)}=\left(A(j-k,j),A(j-k+1,j),A(j-k+2,j),\dots,A(j,j)\right)\in \{0,1\}^{k+1}.\]

It follows from Lemma~\ref{lemmaNextLine} that $f^{(k+1)},f^{(k+2)},f^{(k+3)},\dots$ is a time-homogeneous irreducible Markov-chain. Moreover,
\begin{align*}&\frac{|\{(b,d)\in \mathcal{P}\cap [1,n]^2\,:\,d-b=k\}|}{n}\\&\qquad=\frac{|\{j\in \{k+1,k+2,\dots,n\}\,:\,j-b_j=k\}|}n\\&\qquad
=\frac{\mathbbm{1}(b_{k+1}=1)+|\{j\in \{k+2,k+3,\dots,n\}\,:\,f^{(j-1)}_2=0\text{ and }f^{(j)}_1=1\}|}n.
\end{align*}
Thus, it follows easily from the law of large numbers for Markov-chains  that 
\[\frac{|\{(b,d)\in \mathcal{P}\cap [1,n]^2\,:\,d-b=k\}|}{n}\text{ converges almost surely to a constant.}\]

By equation~\eqref{eqExpLifetime}, we see that this constant can only be $\mathbb{P}(X=k)$.

\section{Fluctuations of the persistent Betti numbers -- The proof of Theorem~\ref{thmfluc}}

The following well-known lemma can be found in \cite{fulman2015stein}.
\begin{lemma}\label{deterministicsizelimit}Let $u\in\mathbb{Z}$, and let $H_n$ be a uniform random $(n+u)\times n$ matrix over $\mathbb{F}_q$. 
\begin{enumerate}[(a)]
 \item If $u\ge 0$, then $\dim\ker H_n$ converges in distribution to $J_u$.
 \item If $u< 0$, then $\dim\ker H_n$ converges in distribution to $|u|+J_{|u|}$.
\end{enumerate}
\end{lemma}
\begin{lemma}\label{randomnondegenerate}
 Let $c_n$ be a sequence of positive integers converging to $\infty$. Let $K_n$ and $K_n'$ be independent $\mathbb{Z}_+$-valued random variables. Assume that there are two independent $\mathbb{Z}$-valued random variables $K_\infty$ and $K_\infty'$ such that $K_n-c_n$ converges in distribution to $K_\infty$, and $K_n'-c_n$ converges in distribution to $K_\infty'$. Let $H_n$ be a uniform random $K_n'\times K_n$ matrix over~$\mathbb{F}_q$. Then $\dim\ker H_n$ converges in distribution to
\[\max(0,-(K_\infty'-K_\infty))+J_{|K_\infty'-K_\infty|},\]
 where $J_0,J_1,\dots$ are independent from $(K_\infty,K_\infty')$. 
\end{lemma}
\begin{proof}
 This is a straightforward corollary of Lemma~\ref{deterministicsizelimit}.
\end{proof}

\begin{lemma}\label{lemmadegenerate}
Let $K_n$ and $K_n'$ be $\mathbb{Z}_+$-valued random variables, and let $H_n$ be a uniform random $K_n'\times K_n$ matrix over $\mathbb{F}_q$.
\begin{enumerate}[(a)]
 \item \label{lemmadegeneratea}Assume that $K_n'-K_n$ converges to $\infty$ in probability. Then 
\[\lim_{n\to\infty}\mathbb{P}(\dim\ker H_n=0)=1.\] 
 \item \label{lemmadegeneratb} Assume that $K_n'-K_n$ converges to $-\infty$ in probability. Then 
\[\lim_{n\to\infty}\mathbb{P}(\dim\ker H_n=K_n-K_n')=1.\]
\end{enumerate}
\end{lemma}
\begin{proof}
 First, we prove part (a). If $H$ is a uniform random $n\times m$ matrix, then by Markov's inequality, we have
\[\mathbb{P}(\dim\ker H>0)\le \min\left(1,\mathbb{E}|\ker H\setminus\{0\}|\right)= \min\left(1,(q^m-1)q^{-n}\right)\le \min\left(1,q^{m-n}\right).\]
 Thus,
\[\lim_{n\to\infty}\mathbb{P}(\dim\ker H_n>0)\le \min\left(1,\mathbb{E}|\ker H_n\setminus\{0\}|\right)= \lim_{n\to\infty}\mathbb{E}\min\left(1,q^{K_n-K_n'}\right)=0.\]
 
 Part (b) follows by applying part (a) to $H_n^T$. 
\end{proof}

The next lemma can be basically found in \cite{van2026rank}, some care is needed to transform the parameters, since \cite{van2026rank} considers strictly triangular matrices. 

\begin{lemma}\label{LemmaRVP}
 Let $n_1<n_2<\cdots$ be a sequence of positive integers such that the fractional parts $\{-\log_q(n_i)\}$ converge to $\zeta$ as $i\to\infty$. Then
\[\dim\ker L_{n_i}-\lfloor \log_q(n_i)+\zeta\rceil\]
 converges to $\mathcal{L}_{1,q^{-1},q^{-1-\zeta}}$ in distribution as $i\to\infty$.
\end{lemma}

Let $K_n$ and $K_n'$ be independent random variables such that $K_n$ has the same distribution as $\dim\ker L_{r_n}$ and $K_n'$ has the same distribution as $\dim\ker L_{t_n}$. Let $H_n$ be uniform random $K_n'\times K_n$ matrix over $\mathbb{F}_q$. By Lemma~\ref{perbettiasrandommatrices}, we see that $\beta^{r_n,r_n+t_n}$ has the same distribution as $\dim\ker H_n$. By Lemma~\ref{LemmaRVP}, we see that
\begin{equation}\label{eqKnFluc}K_n-\lfloor \log_q(r_n)+\zeta\rceil
 \text{ converges to $\mathcal{L}$.}\end{equation}

\medskip
\textbf{The proof of part \eqref{fluc1} of Theorem~\ref{thmfluc}}

In this case, $K_n'\le t$ and $K_n$ converges to $\infty$ in probability by \eqref{eqKnFluc}, so $K_n'-K_n$ tends to~$-\infty$ in probability. Thus, part~\eqref{lemmadegeneratb} of Lemma~\ref{lemmadegenerate} can be applied to give that
\[\lim_{n\to\infty}\mathbb{P}(\dim\ker H_n=K_n-K_n')=1.\]
Thus, $\beta_{r_n,r_n+t}-\lfloor \log_q(r_n)+\zeta\rceil$ has the same limiting distribution as $K_n-\lfloor \log_q(r_n)+\zeta\rceil-K_n'$, which has the same limiting distribution as $\mathcal{L}-D$.

\medskip
\textbf{The proof of part \eqref{fluc2} of Theorem~\ref{thmfluc}}

By Lemma~\ref{LemmaRVP}, we see that
\[K_n'-\lfloor \log_q(t_n)+\zeta'\rceil
 \text{ converges to $\mathcal{L}'$.}\]

Combining this with \eqref{eqKnFluc}, we see that the sequence $K_n'-K_n-(\log_q(t_n)-\log_q(r_n))$ is tight. Since $\log_q(t_n)-\log_q(r_n)\to-\infty$, it follows that $K_n'-K_n$ converges to $-\infty$
 in probability. As in the previous case, this implies that $\beta^{r_n,r_n+t_n}-\left(\lfloor \log_q(r_n)+\zeta\rceil-\lfloor \log_q(t_n)+\zeta'\rceil\right)$ has the same limiting distribution as $(K_n-\left(\lfloor \log_q(r_n)+\zeta\rceil)-(K_n'-\lfloor \log_q(t_n)+\zeta'\rceil\right)$, which has the same limiting distribution as $\mathcal{L}-\mathcal{L}'.$

\medskip
\textbf{The proof of part \eqref{fluc3} of Theorem~\ref{thmfluc}}

Under the assumption $\log_q(t_n)-\log_q(r_n)\to\gamma$, we have that $\{-\log_q(t_n)\}$ converges to~$\zeta'$. Thus, by Lemma~\ref{LemmaRVP}, we see that
\[K_n'-\lfloor \log_q(t_n)+\zeta'\rceil
 \text{ converges to $\mathcal{L}'$.}\]
Observing that $\gamma+\zeta'-\zeta\in\mathbb{Z}$, we see that  for large enough $n$, we have
\begin{align*}\lfloor \log_q(t_n)+\zeta'\rceil&=\lfloor \log_q(r_n)+\zeta +(\log_q(t_n)-\log_q(r_n)-\gamma)+(\gamma+\zeta'-\zeta)\rceil\\&=\lfloor \log_q(r_n)+\zeta\rceil+\gamma+\zeta'-\zeta.\end{align*}
Thus,
\[K_n'-\lfloor \log_q(r_n)+\zeta\rceil
 \text{ converges to $\mathcal{L}'+\gamma+\zeta'-\zeta$.}\]
Thus, with the choice of $c_n=\lfloor \log_q(r_n)+\zeta\rceil$ one can apply Lemma~\ref{randomnondegenerate} to obtain the statement.

\medskip
\textbf{The proof of part \eqref{fluc4} of Theorem~\ref{thmfluc}}

By Lemma~\ref{LemmaRVP}, we see that $K_n-\log_q(r_n)$ and $K_n'-\log_q(t_n)$ are both tight sequences. Since $\log_q(t_n)-\log_q(r_n)\to\infty$, it follows that $K_n'-K_n$ converges to $\infty$ in probability. 

Thus, by part~\eqref{lemmadegeneratea} of Lemma~\ref{lemmadegenerate}, we see that
\[\lim_{n\to\infty}\mathbb{P}(\dim\ker H_n=0)=1.\]
Therefore, the statement follows.

\bigskip

\textbf{Open question:} It would be interesting to find the joint distribution of fluctuations of a finite collection of persistent Betti numbers.

\bibliography{references}

\bibliographystyle{plain}

\bigskip

\noindent Andr\'as M\'esz\'aros, \\
HUN-REN Alfr\'ed R\'enyi Institute of Mathematics, \\Budapest, Hungary,\\ {\tt meszaros@renyi.hu}

\end{document}